\RequirePackage{fix-cm}
\documentclass[smallextended]{svjour3}       
\smartqed  
\usepackage{graphicx}
\usepackage{amsmath,amssymb,booktabs,mathtools}
\usepackage{dsfont}
\usepackage{enumitem}
\usepackage{hyperref}
\usepackage{caption}

\allowdisplaybreaks

\newcommand{\ccirc}{\mathbin{\mathchoice
  {\xcirc\scriptstyle}
  {\xcirc\scriptstyle}
  {\xcirc\scriptscriptstyle}
  {\xcirc\scriptscriptstyle}
}}
\newcommand{\xcirc}[1]{\vcenter{\hbox{$#1\circ$}}}
\newcommand{\E}{\mathbb{E}}
\newcommand{\p}{\mathbb{P}}
\newcommand{\LL}{L_{\pi}}
\newcommand{\lx}{L_{\xi}}
\newcommand{\ft}{f_{R_{\delta}}}
\newcommand{\m}{\mu_{n, q}}
\newcommand{\n}{\mu_{n, q_n}}
\newcommand{\inv}{l}
\newcommand{\IV}{\text{Inv}}
\newcommand{\var}{\text{Var}}
\newcommand{\cov}{\text{Cov}}

\newcommand{\bs}{\boldsymbol}

\begin{document}

\title{The Limit of the Empirical Measure of the Product of Two Independent Mallows Permutations
}

\titlerunning{The Empirical Measure of the Product of Mallows Permutations}        

\author{Ke Jin         
}


\institute{Ke Jin \at
              Department of Mathematical Sciences, University of Delaware \\
              \email{kejin@udel.edu}           
}

\date{Received: date / Accepted: date}

\maketitle

\begin{abstract}
The Mallows measure is a probability measure on $S_n$ where the probability of a permutation $\pi$ is proportional to $q^{l(\pi)}$ with $q > 0$ being a parameter and $l(\pi)$ the number of inversions in $\pi$. We show the convergence of the random empirical measure of the product of two independent permutations drawn from the Mallows measure, when $q$ is a function of $n$ and $n(1-q)$ has limit in $\mathbb{R}$ as $n \to \infty$.
\end{abstract}

\section{Introduction}\label{S1}

\subsection{Background}
\begin{definition}
Given $\pi \in S_n$, the inversion set of $\pi$ is defined by
\[
\IV(\pi) \coloneqq \{ (i, j) : 1 \le i < j \le n \text{ and } \pi(i) > \pi(j) \},
\]
and the inversion number of $\pi$, denoted by $\inv(\pi)$, is defined to be the cardinality of $\IV(\pi)$.
\end{definition}
The Mallows measure on $S_n$ is introduced by Mallows in \cite{mallows1957}. For $q > 0$, the $(n, q)$ - Mallows measure on $S_n$ is given by
\[
\m(\pi) \coloneqq \frac{q^{\inv(\pi)}}{Z_{n, q}}, \quad \text{where }\  Z_{n, q} = \prod_{i = 1}^n \frac{1 - q^i}{1 - q}.
\]
Here $Z_{n, q}$ is the normalizing constant, which has an explicit form (see, e.g., \cite{Stanley} Corollary 1.3.13). In other words, under the Mallows measure with parameter $q>0$, the probability of a permutation $\pi$ is proportional to $q^{\inv(\pi)}$. 

Mallows measure has been used in modeling ranked and partially ranked data (see, e.g., \cite{CR}, \cite{FV}, \cite{Ma}). In \cite{CR}, Critchlow provides several examples where Mallows model gives a good fit to ranking data.
In Starr's paper \cite{Starr}, he proves the following result showing that the random empirical measure induced from a Mallows permutation converges in probability to a non-random probability measure with an explicit density.


\begin{theorem}[S.Starr]\label{T1}
Suppose that $(q_n)_{n=1}^{\infty}$ is a sequence such that the limit $\beta = \lim_{n\to\infty} n (1-q_n)$ exists. For any $\epsilon >0$ and any continuous function $f: [0, 1]\times[0,1] \to \mathbb{R}$,
\[
\lim_{n\to \infty} \mu_{n, q_n}
\bigg(\bigg|\frac{1}{n}\sum_{i=1}^{n}f\Big(\frac{i}{n}, \frac{\pi(i)}{n}\Big) -
 \int_{[0,1]\times[0,1]} f(x,y)u(x,y,\beta)\ dxdy\bigg|>\epsilon \bigg) = 0,
\]
where
\begin{equation}\label{eq:t1}
u(x, y, \beta) = \frac{(\beta/2) \sinh(\beta/2)}{\left(e^{\beta/4} \cosh(\beta[x-y]/2)-e^{-\beta/4}\cosh(\beta[x+y-1]/2)\right)^2},
\end{equation}
if $\beta \neq 0$, and $u(x, y, 0) \coloneqq 1$.
\end{theorem}
The author proves the theorem above by making use of the mean field theory and evaluates the density of the limiting distribution as the solution to an integrable PDE. We do not think this approach applies in our case, since the Hamiltonian is not of mean-field type. In this paper, we establish a similar result for the empirical measure of the product of two independent Mallows permutations. Here the product of two permutations is taken within the symmetric group $S_n$, and our proof takes an entirely different approach.

\subsection{Results}

Our first result is about the distribution of $\frac{\pi(i)}{n}$ in the regime of the Mallows measure where $\lim_{n \to \infty}n(1-q_n)$ exists. It says that the distribution of 
$\frac{\pi(i)}{n}$ approaches towards the measure with density $u\left(\frac{i}{n}, y, \beta \right)$ uniformly on $i \in [n]$. Here the `approach' is in the sense that, given any continuous function $f$ on $[0, 1]$, the expectation of $f$ with respect to the empirical measure of $\frac{\pi(i)}{n}$ converges uniformly to the expectation of $f$ with respect to the probability with density $u\left(\frac{i}{n}, y, \beta \right)$. Moreover, the covariance of $f(\frac{\pi(i)}{n})$ and $f(\frac{\pi(j)}{n})$ converges to 0 uniformly on all pairs $(i, j)$ where $i \neq j$.

\begin{theorem}\label{M0}
Suppose that $\{q_n\}_{n=1}^{\infty}$ is a sequence such that $\lim\limits_{n\to\infty} n (1-q_n) = \beta \in \mathbb{R}$. For any continuous function $f : [0, 1] \longrightarrow \mathbb{R}$, we have
\begin{equation}\label{eq:M01}
\textstyle\lim\limits_{n \to \infty}\max\limits_{i \in [n]}
  \left| \n \left( f\left(\frac{\pi(i)}{n}\right) \right) - \int_0^1 f(y)\cdot u\left(\frac{i}{n}, y, \beta \right)\, dy \right| = 0,
\end{equation}
and
\begin{equation}\label{eq:M02}
\textstyle\lim\limits_{n \to \infty}\max\limits_{\substack{ i \neq j\\
                                   i, j \in [n]}}
    \left| \cov_n\Big(f\big(\frac{\pi(i)}{n}\big), f\big(\frac{\pi(j)}{n}\big)\Big) \right| = 0.
\end{equation}
Here $u(x, y, \beta)$ is defined in (\ref{eq:t1}), and
\begin{align*}
&\textstyle\cov_n\Big(f\big(\frac{\pi(i)}{n}\big), f\big(\frac{\pi(j)}{n}\big)\Big) \coloneqq\\
&\textstyle\qquad\qquad \n \left(f\big(\frac{\pi(i)}{n}\big)f\big(\frac{\pi(j)}{n}\big)\right) -
     \n \left(f\big(\frac{\pi(i)}{n}\big)\right) \n \left(f\big(\frac{\pi(j)}{n}\big)\right).
\end{align*}
\end{theorem}

Theorem \ref{M0} is a major step in proving Theorem \ref{M2}, which shows the convergence of the empirical measure defined by the product of two independent Mallows distributed permutations. 

\begin{theorem}\label{M2}
Suppose that $\{q_n\}_{n=1}^{\infty}$ and $\{q'_n\}_{n=1}^{\infty}$ are two sequences such that $\lim_{n\to\infty} n (1-q_n) = \beta$  and
$\lim_{n\to\infty} n (1-q'_n) = \gamma$, with $\beta$, $\gamma \in \mathbb{R}$. Let $\p_n$ denote the probability measure on $S_n \times S_n$
such that $\p_n\big((\pi, \tau)\big) = \n(\pi) \cdot \mu_{n, q'_n}(\tau)$, i.~e.~$\p_n$ is the product measure of $\n$ and $\mu_{n, q'_n}$.
Let $\tau \ccirc \pi$ denote the product of $\tau$ and $\pi$ in $S_n$ with $\tau \ccirc \pi(i) = \tau(\pi(i))$.
Then, for any $\epsilon > 0$,
\[
\lim_{n \to \infty}\p_n\left(\left|\, \frac{1}{n}\sum_{i=1}^{n}f\left(\frac{i}{n}, \frac{\tau \ccirc \pi(i)}{n}\right) -
  \int_0^1\int_0^1 f(x, y) \rho(x, y)\,dxdy\, \right| > \epsilon \right) = 0
\]
for every continuous function $f: [0,1]\times[0,1] \rightarrow \mathbb{R}$, with
\begin{equation}\label{eq:m2a}
\rho(x, y) \coloneqq \int_{0}^{1} u(x, t, \beta)\cdot u(t, y, \gamma)\,dt,
\end{equation}
where $u(x, y, \beta)$ is defined in (\ref{eq:t1}).
\end{theorem}

Theorem \ref{M2} is used in our proof (see \cite{Ke}) of a weak law of large numbers for the length of the longest common subsequence of two independent Mallows permutations. Theorem \ref{M2} is also of interest in its own right because it provides us a nontrivial example for a more general question: if we have two independent sequences of random permutations $\{\pi_n\}$ and $\{\tau_n\}$ whose limiting empirical measures are known, under what condition does the limiting empirical measure of $\{\tau_n\ccirc\pi_n\}$ exist with density of a similar form as $\rho$ in (\ref{eq:m2a})?

The paper is organized as follows. In Section 2, we introduce and prove Lemma \ref{L18} and Lemma \ref{L13}. In Section 3 and Section 4, we show Theorem \ref{M0} and Theorem \ref{M2} respectively using those two lemmas established in Section 2. We conclude the paper by discussing some future works.

\section{Two Key Lemmas}
In this section we introduce the following two lemmas which play the key role in proving the main theorems. The proofs presented in this section are largely independent of the following sections. With these two lemmas in mind, readers can go through the proofs of the main theorems without trouble.

\begin{lemma}\label{L18}
Suppose $A = [y_1, y_2] \subset [0,1]$. For any $\beta \in \mathbb{R}$ and any sequence $\{q_n \}$ such that $q_n > 0$ and $\lim_{n \to \infty} n (1 - q_n) = \beta$,
\[
\textstyle\lim\limits_{n \to \infty}\max\limits_{i \in [n]}
  \left| \n \left(\mathds{1}_A\left(\frac{\pi(i)}{n}\right) \right) - \int_{y_1}^{y_2} u\left(\frac{i}{n}, y, \beta \right)\, dy \right| = 0.
\]
\end{lemma}


\begin{lemma}\label{L13}
Suppose $A = [y_1, y_2] \subset [0, 1]$ and $B = [y_3, y_4] \subset [0, 1]$. Given $\beta \in \mathbb{R}$ and any sequence $\{q_n \}$ such that $q_n > 0$ and $\lim_{n \to \infty} n (1 - q_n) = \beta$, define
\begin{align*}
&\textstyle\cov_n\Big(\mathds{1}_A\big(\frac{\pi(i)}{n}\big), \mathds{1}_B\big(\frac{\pi(j)}{n}\big)\Big) \coloneqq\\
&\textstyle\qquad\qquad \n \left(\mathds{1}_A\big(\frac{\pi(i)}{n}\big)\mathds{1}_B\big(\frac{\pi(j)}{n}\big)\right) -
     \n \left(\mathds{1}_A\big(\frac{\pi(i)}{n}\big)\right) \n \left(\mathds{1}_B\big(\frac{\pi(j)}{n}\big)\right).
\end{align*}
Then, we have
\[
\textstyle\lim\limits_{n \to \infty}\max\limits_{\substack{ i \neq j\\
                                   i, j \in [n]}}
    \left| \cov_n\Big(\mathds{1}_A\big(\frac{\pi(i)}{n}\big), \mathds{1}_B\big(\frac{\pi(j)}{n}\big)\Big) \right| = 0.
\]
\end{lemma}

Lemma \ref{L18} states that in the regime of Mallows permutation where $n(1-q_n)$ has limit in $\mathbb{R}$, the probability of $\frac{\pi(i)}{n}$ falling in an arbitrary interval converges to a constant uniformly for $i \in [n]$. Lemma \ref{L13} states that the covariance of $\mathds{1}_A\big(\frac{\pi(i)}{n}\big)$ and $\mathds{1}_B\big(\frac{\pi(j)}{n}\big)$ converges to 0 uniformly on all those pairs such that $i \neq j$. 

The proofs of these two lemmas involves some computation which utilize Theorem \ref{T1} and properties of Mallows permutation. It may be the case that more general tools could be used to establish the uniform convergence of the distribution of $\frac{\pi(i)}{n}$ as well as 
$\cov_n(\frac{\pi(i)}{n}, \frac{\pi(j)}{n})$.

\subsection{Preliminaries}\label{S2}

Let $\mu$ be a probability measure on the Borel $\sigma$-field $\mathcal{B}_{\Sigma}$. We use the convention that $\mu(f) = \int_{\Sigma} f \ d\mu$, for any measurable function $f$.
For any $\pi \in S_n$, let $L_{\pi}$ denote the empirical measure induced by $\pi$, that is,
\begin{equation}\label{eq:em}
L_{\pi}(R) \coloneqq \frac{1}{n}\sum_{i = 1}^{n} \mathds{1}_R\left( \frac{i}{n}, \frac{\pi(i)}{n} \right),
\end{equation}
for any $R \in \mathcal{B}_{[0, 1]\times[0,1]}$. Here $\mathds{1}_R(x, y)$ denotes the indicator function of $R$. Hence, for any measurable function $f$,
\[
L_{\pi}(f) = \frac{1}{n}\sum_{i=1}^{n}f\left(\frac{i}{n}, \frac{\pi(i)}{n}\right).
\]

For  any $\pi \in S_n$,  Let $\bs{z}(\pi) \coloneqq \{(\frac{i}{n}, \frac{\pi(i)}{n})\}_{i \in [n]}$ denote the set of $n$ points in $[0,1] \times [0,1]$ defined by $\pi$. Conversely, for any $n$ points $V \coloneqq\{(x_i, y_i)\}_{i \in [n]}$ such that $i \neq j$ implies $x_i \neq x_j$ and $y_i \neq y_j$, we can define a permutation $\pi \in S_n$ as follows. Without loss of generality, assuming $x_1 < \cdots < x_n$, define
\[
\pi(i) \coloneqq |\{j \in [n] : y_j \le y_i\}|.
\]
We will use $\Phi(V)$ to denote the permutation induced by $V$ as above. Similarly, we define the number of inversions of a collection points as follows,
\[
\inv(V) \coloneqq |\{(i, j) : (x_i - x_j)(y_i - y_j) < 0 \text{ and }  i < j \}|.
\]
Note that the definition of the number of inversions of a collection of points is consistent with the definition of inversion of permutation in the sense that, for any $\pi \in S_n$,
\[
\inv(\pi) = \inv(\bs{z}(\pi)) \quad \text{ and }\quad \inv(V) = \inv\left(\Phi(V)\right).
\]
\begin{definition}
For any $\pi \in S_n$ and $i \in [n]$, define 
\[
\textstyle\pi^{(i)} \coloneqq \Phi\left(\Big\{\left(\frac{j}{n}, \frac{\pi(j)}{n}\right): j \neq i\Big\}\right) \quad \text{and} \quad Q(\pi, i) \coloneqq \{ \tau \in S_n: \tau^{(i)} = \pi^{(i)} \}.
\]
In other words, $\pi^{(i)}$ denotes the permutation in $S_{n-1}$ which is induced from $\pi$ at those indices other than $i$, and $Q(\pi, i)$ contains those permutations in $S_n$ each of which has the same relative ordering as $\pi$ at those indices other than $i$.
\end{definition}
The definition above is best understood when we represent a permutation by a grid of tiles. Specifically, for any $\pi \in S_n$, define an $n\times n$ grid of tiles such that the tile at $j$-th row and $i$-th column is black if only if $\pi(i) = j$. Here we index the row number from bottom to top, i.e.\,the bottom row is indexed as the first row. For example, the grid representations of $\pi = (4,1,7,3,6,2,5)$ and $\pi^{(4)} = (3,1,6,5,2,4)$ are shown in the following figures.
\begin{figure}[ht]
\hfill
\begin{minipage}[b]{0.4\textwidth} \centering
\includegraphics[width = 3.5cm]{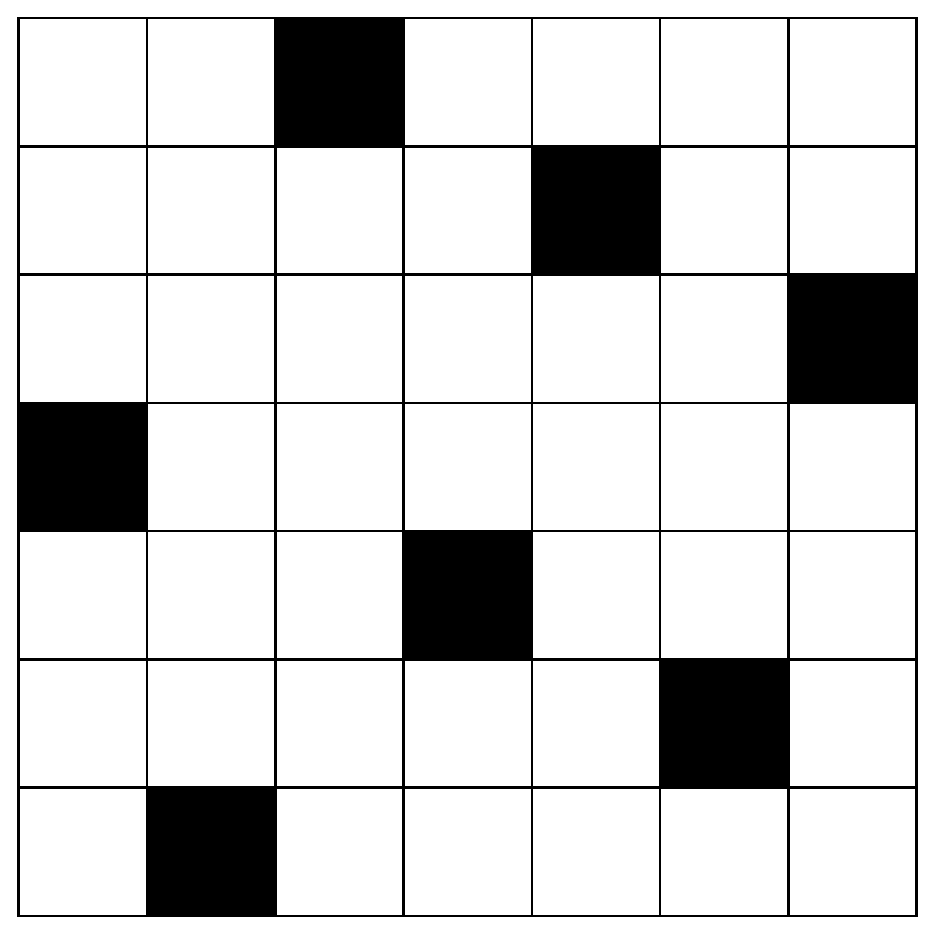}
\caption*{\small$\pi = (4,1,7,3,6,2,5)$}
\end{minipage}
\hfill
\begin{minipage}[b]{0.4\textwidth} \centering
\includegraphics[width = 3cm]{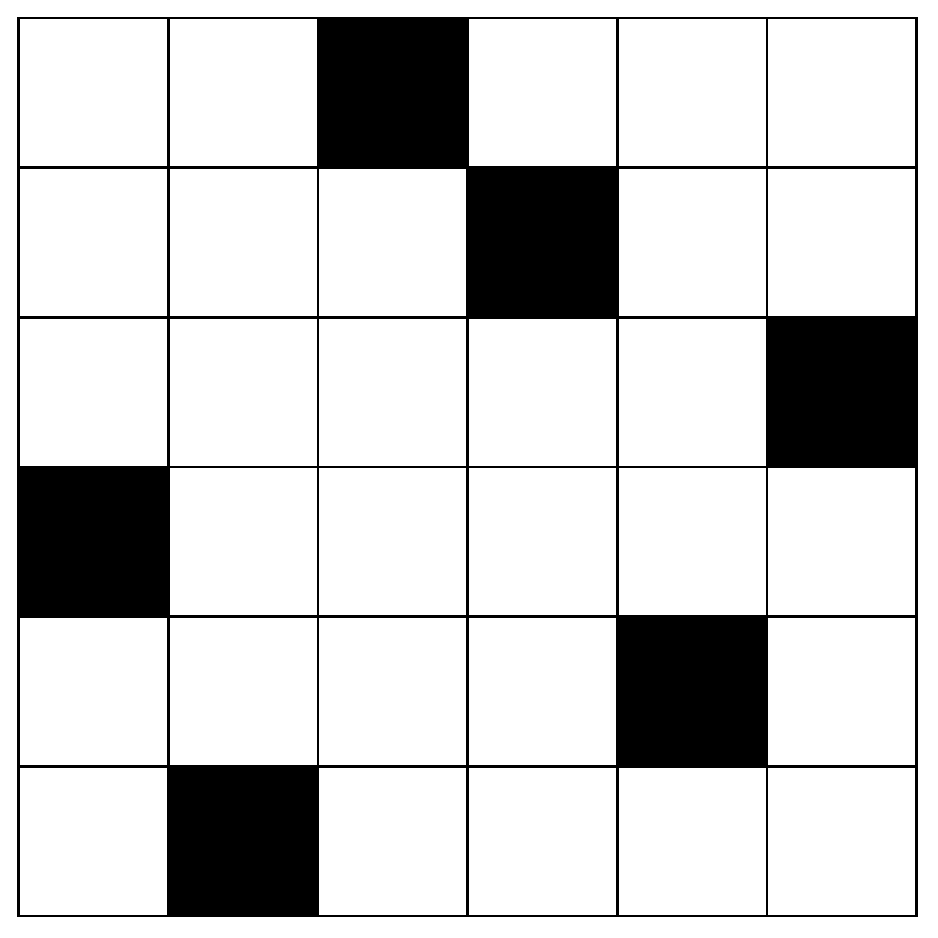}
\caption*{\small$\pi^{(4)} = (3,1,6,5,2,4)$}
\end{minipage}
\hfill
\end{figure}

Note that the grid representation of $\pi^{(i)}$ can be easily obtained by deleting the $i$-th column and $\pi(i)$-th row from the grid of $\pi$. Also, the grid representations of those permutations other than $\pi$ in $Q(\pi, i)$ can be obtained by removing and reinserting the $\pi(i)$-th row into the grid of $\pi$. For example, it can be easily verified that $\tau = (3,1,7,6,5,2,4) \in Q(\pi, 4)$. The grid representation of $\tau$ can be obtained by removing the third row from the grid of $\pi$ and reinserting it between the sixth row and seventh row of the grid of $\pi$ (see Figure \ref{fg:2}).
\begin{figure}[h]
\hfill
\begin{minipage}[b]{0.4\textwidth} \centering
\includegraphics[height = 3.5cm]{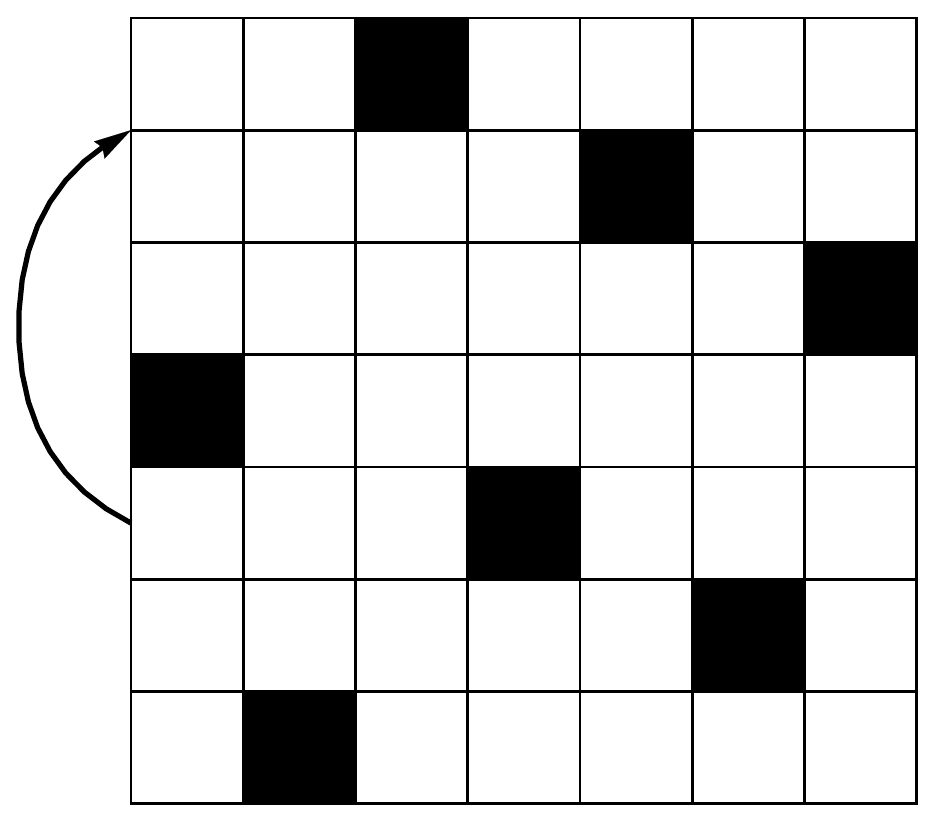}
\caption*{\small$\pi = (4,1,7,3,6,2,5)$}
\end{minipage}
\begin{minipage}[t]{0.1\textwidth} \centering
\vspace{-2.3cm}
$\Longrightarrow$
\vspace{2cm}
\caption{}\label{fg:2}
\end{minipage}
\begin{minipage}[b]{0.4\textwidth} \centering
\includegraphics[height = 3.5cm]{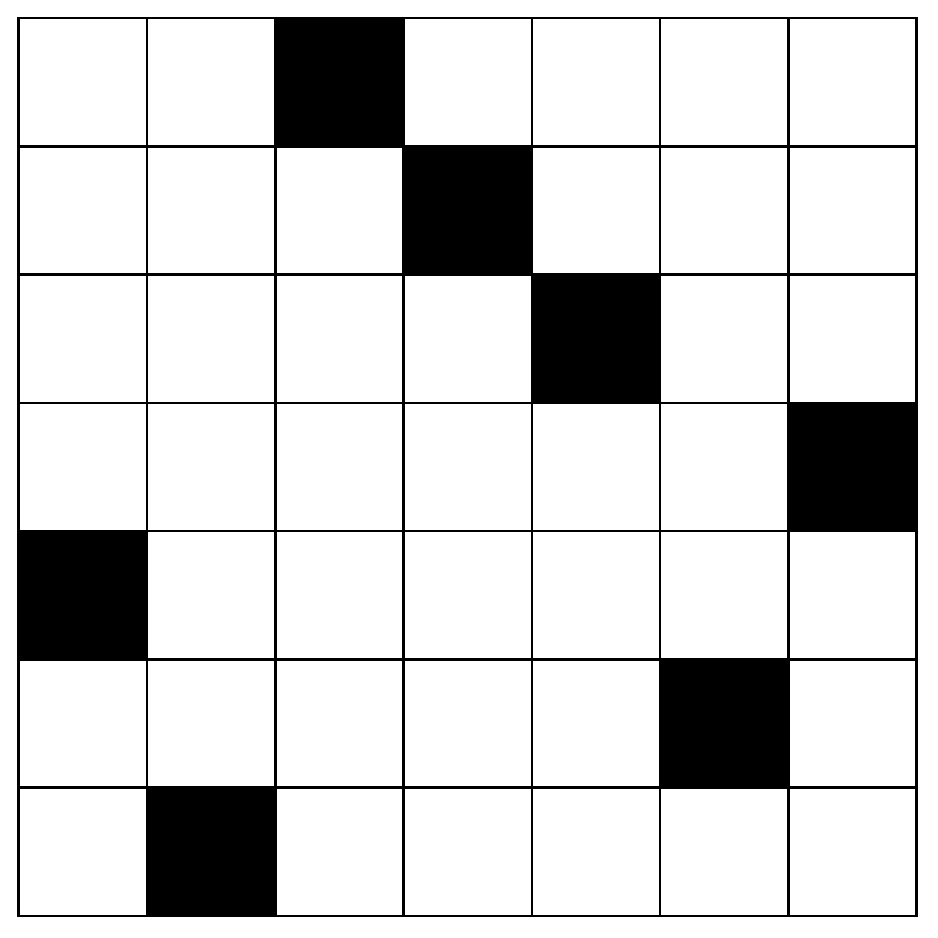}
\caption*{\small$\tau = (3,1,7,6,5,2,4)$}
\end{minipage}
\hfill
\end{figure}

From this definition, it can be seen that $|Q(\pi, i)| = n$ for any $\pi \in S_n$. Also, for any $\pi, \tau \in S_n$, we have either
$Q(\pi, i) = Q(\tau, i)$ or $Q(\pi, i) \cap Q(\tau, i) = \varnothing$.

\begin{proposition}\label{P1}
For any $\pi, \tau\in Q(\pi, i)$, with $\pi(i) = j < k = \tau(i)$, it holds that
\begin{align*}
\inv(\tau) - \inv(\pi) &= |\{ t > i : j+1 \le \pi(t) \le k \}| - |\{ t < i : j+1 \le \pi(t) \le k \}|\\
                       &= |\{ t > i : j \le \tau(t) \le k-1 \}| - |\{ t < i : j \le \tau(t) \le k-1 \}|.
\end{align*}
\end{proposition}
\begin{proof}
This result can be easily seen from the grid representations of $\pi$ and $\tau$. Note that an inversion in a permutation corresponds to a pair of black tiles such that one tile is located to the southeast of the other. Hence, by the discussion above, we only need to count the change of the number of those pairs when we reinsert the $j$-th row of $\pi$'s grid to get the grid form of $\tau$. Specifically, we only need to consider those pairs which contain the black tile on the $i$-th column.
\begin{figure}[h]
\hfill
\begin{minipage}[b]{0.4\textwidth} \centering
\includegraphics[height = 3.5cm]{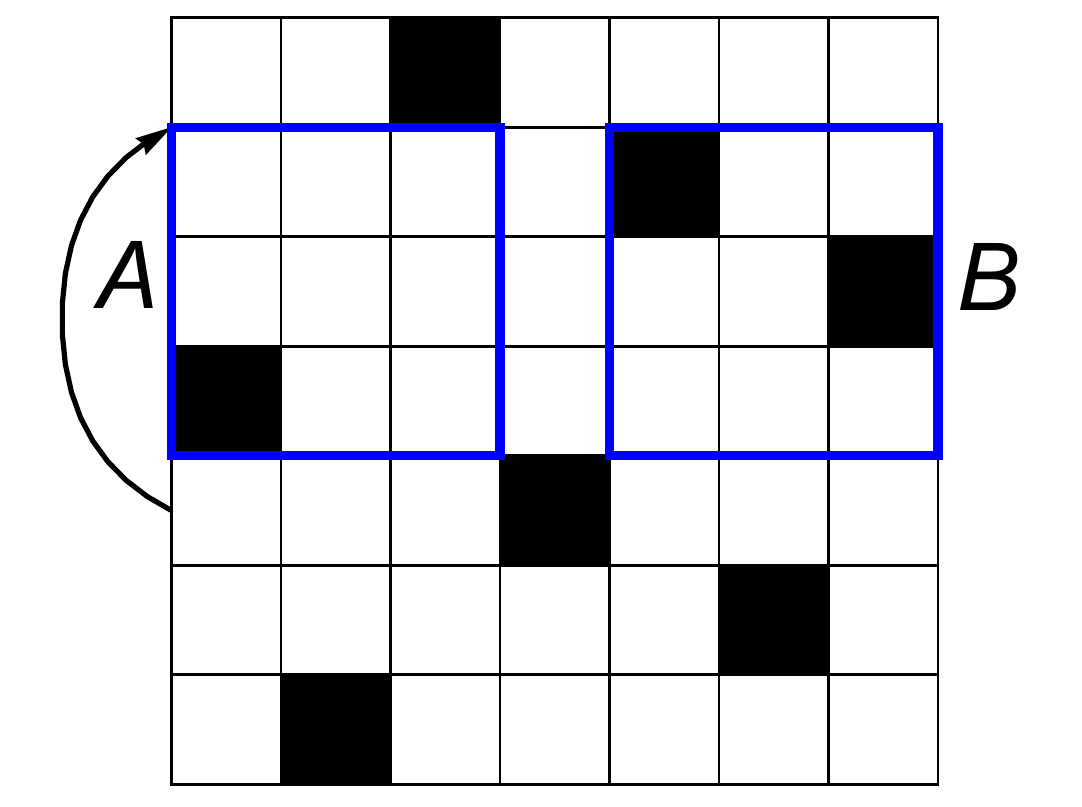}
\caption*{\small$\pi = (4,1,7,3,6,2,5)$}
\end{minipage}
\begin{minipage}[t]{0.1\textwidth} \centering
\vspace{-2.3cm}
$\Longrightarrow$
\vspace{2cm}
\caption{}\label{fg:3}
\end{minipage}
\begin{minipage}[b]{0.4\textwidth} \centering
\includegraphics[height = 3.5cm]{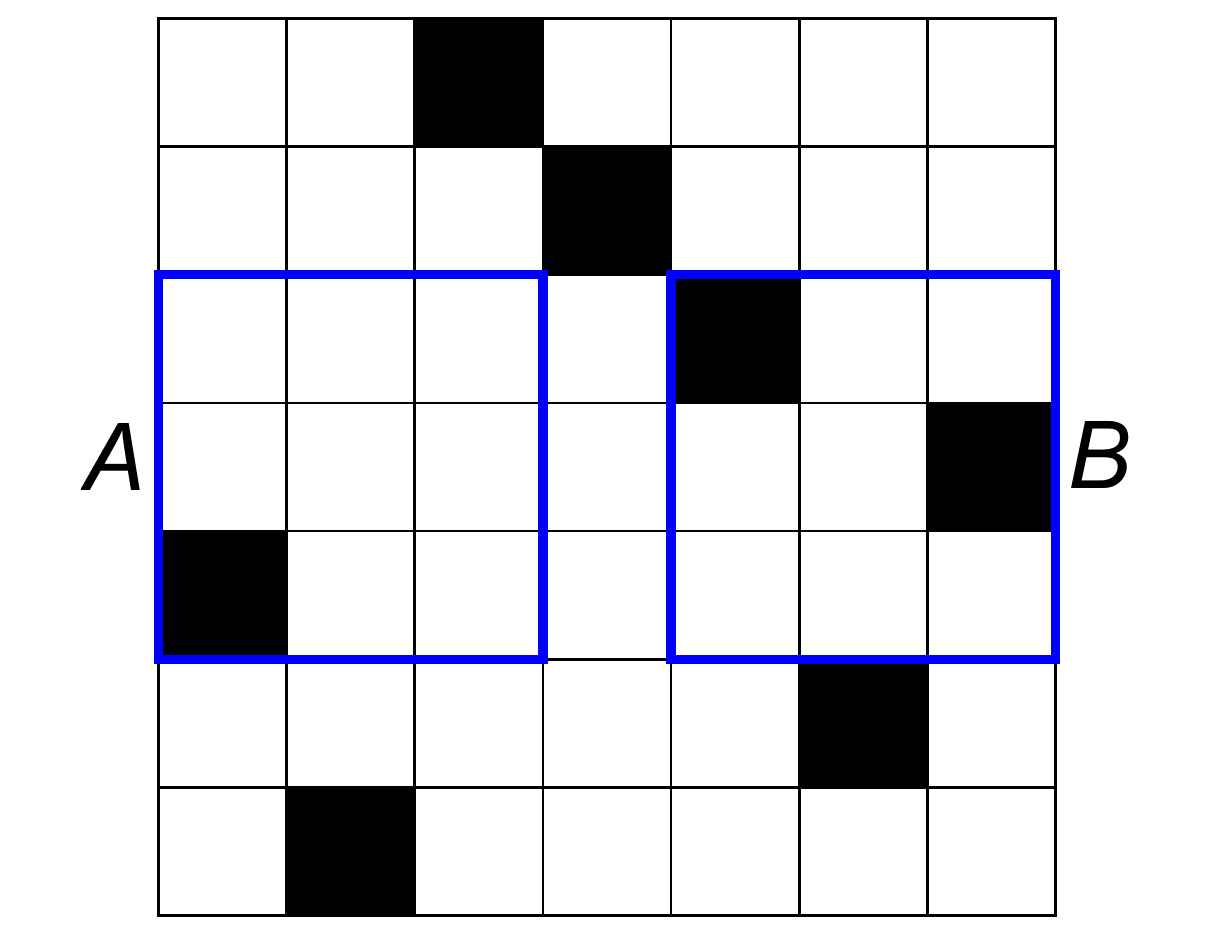}
\caption*{\small$\tau = (3,1,7,6,5,2,4)$}
\end{minipage}
\hfill
\end{figure}

Taking the same example above, $\inv(\tau) - \inv(\pi)$ is equal to the difference of the number of black tiles within the rectangles $A$ and $B$(see Figure \ref{fg:3}). This is because, each of those black tiles in rectangle $A$ forms an inversion with the black tile in the fourth column in the grid representation of $\pi$ but not in that of $\tau$, whereas the opposite holds for those black tiles in the rectangle $B$.
\end{proof}

\begin{definition}
For any $\pi \in S_n$, let $\pi^r \in S_n$ denote the reversal of $\pi$ which is defined by $\pi^r(i) \coloneqq \pi(n+1-i)$ for any $i \in [n]$. Let $\pi^{-1}$ denote the inverse of $\pi$ in the symmetric group $S_n$.
\end{definition}

One property of Mallows permutation is the following proposition (cf. Lemma 2.2 in \cite{Naya}).
\begin{proposition}\label{P3}
For any $n \ge 1$ and $q > 0$, if $\pi \sim \m$ then $\pi^r \sim \mu_{n, 1/q}$ and $\pi^{-1} \sim \m$.
\end{proposition}

\subsection{One dimension analog of Theorem \ref{T1}}

The following lemma is the one dimensional analog of Theorem \ref{T1}. It says that, in the regime of Mallows measure where $\lim_{n \to \infty}n(1-q_n)$ exists, the distribution of $\frac{\pi(a_n)}{n}$ converges in distribution to a probability measure with explicit density, where $\{a_n\}$ is a sequence of indices such that $\lim_{n \to \infty}\frac{a_n}{n}$ exists.

\begin{lemma}\label{M1}
Suppose that $\{q_n\}_{n=1}^{\infty}$ is a sequence such that $\lim\limits_{n\to\infty} n (1-q_n) = \beta \in \mathbb{R}$, and $\{a_n\}$ is a sequence such that $\lim\limits_{n \to \infty} \frac{a_n}{n} = a$, where $a \in [0, 1]$ and $a_n \in [n]$. Then,
\[
\n \left( \frac{\pi(a_n)}{n} \in (\cdot) \right) \overset{d}{\longrightarrow} v.
\]
Here $\overset{d}{\longrightarrow}$ denotes convergence in distribution and $v$ is the probability measure on $[0,1]$ with density $f(y) = u(a, y, \beta)$ where $u(x, y, \beta)$ is as defined in (\ref{eq:t1}).
\end{lemma}

We will sometimes omit the third argument and simply use $u(x, y)$ to denote $u(x, y, \beta)$, if no confusion arises from the context. We use the symbol $u_{\beta}$ or $u$ to denote the measure on $[0, 1]\times[0,1]$ which has density $u(x, y, \beta)$ with respect to the Lebesgue measure $\lambda$.

To prove Lemma \ref{M1}, we show that any convergent subsequence of the empirical measures $\{\frac{\pi(a_n)}{n} \}$ has limiting density $u(a, y, \beta)$ and the theorem follows from the standard result of convergence of measures (Theorem \ref{thm:AshDade}). It is unknown to us whether Lemma \ref{M1} can be obtained directly from Theorem \ref{T1}. In the remainder of this section, we prove a sequence of technical lemmas to show Lemma \ref{M1}. The following lemma says that the result of Theorem \ref{T1} also holds when $f$ is an indicator function of any rectangle.


\begin{lemma}\label{L1}
Under the same conditions as in Theorem \ref{T1}, for any $\epsilon >0$,
\[
\lim_{n\to \infty} \mu_{n, q_n}
\bigg(\bigg|\frac{1}{n}\sum_{i=1}^{n}\mathds{1}_R\Big(\frac{i}{n}, \frac{\pi(i)}{n}\Big) -
 \int_{R} u(x,y)\ dxdy\bigg|>\epsilon \bigg) = 0,
\]
for any $R = [x_1, x_2]\times[y_1, y_2] \subset [0, 1]\times[0, 1]$.
\end{lemma}
\begin{proof}
First we show that for any $R = [x_1, x_2]\times[y_1, y_2]$ and any $\epsilon >0$, when $n$ is sufficiently large,
\begin{equation}
\LL(R) < \min(x_2-x_1, y_2-y_1) + \frac{\epsilon}{24},  \label{eq:L1a}
\end{equation}
for any $\pi \in S_n$. Let $s \coloneqq\min(x_2-x_1, y_2-y_1)$. For any $\pi \in S_n$, we have
\[
\Big|\Big\{i : \Big(\frac{i}{n}, \frac{\pi(i)}{n}\Big) \in R \Big\}\Big| \le n s +1,
\]
since, of the points in $\big\{\big(\frac{i}{n}, \frac{\pi(i)}{n}\big)\big\}$, there is one and only one point on each line $x = \frac{i}{n}$ or $y = \frac{j}{n}$. Hence, $\LL(R) \le s+ \frac{1}{n}$ for any $\pi \in S_n$. We can choose $n$ large enough such that $\frac{1}{n} < \frac{\epsilon}{24}$.\\ 
Next, given $\delta >0$, let $R_{\delta} \coloneqq (x_1-\delta, x_2+\delta)\times(y_1 - \delta, y_2+\delta)$. Let $D \coloneqq R_{\delta}-R$. Then, it is easily seen that $D$ can be covered by four rectangles each of whose smaller side is no greater than $\delta$.
For any $\delta > 0$, by Urysohn's lemma (cf. 12.1 in \cite{Royden}), we can choose a continuous function $\ft(x, y)$, such that,
\[
 \begin{cases}
   \ft(x, y)= 1 &\text{if } (x, y) \in R,\\
   \ft(x, y)= 0 &\text{if } (x, y) \notin R_{\delta},\\
   0 \le \ft(x, y) \le 1 &\text{if } (x, y) \in D.
 \end{cases}
\]
By the triangle inequality, we have
\begin{align}
 |&\LL(R) - u(R)|> \epsilon   \label{eq:L1b}\\
 \Rightarrow |&\LL(\ft) - \LL(R)| + |u(\ft) - u(R)| + |\LL(\ft) - u(\ft)| > \epsilon. \nonumber
\end{align}
If we choose $\delta < \frac{\epsilon}{24}$, by (\ref{eq:L1a}), we have,
\[
|\LL(\ft) - \LL(R)| \le \LL(R_{\delta}) - \LL(R) = \LL(D) < 4 \left(\frac{\epsilon}{24} + \frac{\epsilon}{24}\right) = \frac{\epsilon}{3},
\]
for any $\pi \in S_n$, when $n$ is sufficiently large.
Since $u$ is absolutely continuous with respect to the Lebesgue measure, we may choose $\delta$ small enough such that
\[
|u(\ft) - u(R)|\le u(D) < \frac{\epsilon}{3}.
\]
Then by (\ref{eq:L1b}), for sufficiently large $n$,  we have,
\[
|\LL(R) - u(R)|> \epsilon \quad \Rightarrow \quad |\LL(\ft) - u(\ft)| > \frac{\epsilon}{3}.
\]
Thus,
\[
\m\Big(|\LL(R) - u(R)|> \epsilon\Big) \le \m\Big(|\LL(\ft) - u(\ft)| > \frac{\epsilon}{3}\Big).
\]
The lemma follows by Theorem \ref{T1}.
\end{proof}

The following property of Mallows permutations will be used in later proofs. It says that in a Mallows permutation, the relative chance that $\pi(i)$ takes two different values can be bounded in terms of the difference of those two values.

\begin{lemma}\label{L2}
For any $ 1\le i, s, t \le n$ and $q > 0$,
\[
\min( q^d, q^{-d}) \le \frac{\m(\pi(i)=s)}{\m(\pi(i)=t)}\le \max( q^d, q^{-d}),
\]
where $d = |s-t|$.
\end{lemma}

\begin{proof}
Suppose $0< q < 1$. We claim it suffices to show that
\begin{equation}
q \le \frac{\m(\pi(i)=j+1)}{\m(\pi(i)=j)} \le \frac{1}{q}, \label{eq:L2a}
\end{equation}
for any $j \in [n-1]$. This follows since by taking the reciprocal of (\ref{eq:L2a}), we get
\[
q \le \frac{\m(\pi(i)=j)}{\m(\pi(i)=j+1)} \le \frac{1}{q},
\]
and the lemma follows by induction on $d$.\\
Consider the bijection $T_j$ on $S_n$: $\pi \rightarrow (j, j+1) \ccirc \pi$. Here $\ccirc$ denotes the group operator of $S_n$, and $(j, j+1)$ denotes the transposition of $j$ and $j+1$. Specifically, for any $i \in [n]$
\[
T_j(\pi)(i) =
 \begin{cases}
   j &\text{if } \pi(i) = j+1, \\
   j+1 &\text{if } \pi(i) = j, \\
   \pi(i) &\text{otherwise}.
 \end{cases}
\]
From the definition, it is not hard to see that $|\inv(\pi) - \inv(T_j(\pi))|=1 $, for any $\pi \in S_n$. Hence,
\begin{equation}
q \le \frac{\m(T_j(\pi))}{\m(\pi)} \le \frac{1}{q}. \label{eq:L2b}
\end{equation}
Let $A_{i, j} \coloneqq \{\pi \in S_n: \pi(i) = j\}$. For any fixed $i \in [n]$,  $T_j$ is also a bijection of $A_{i, j}$ and $A_{i, j+1}$. Hence,
\begin{equation}
\frac{\m(\pi(i)=j+1)}{\m(\pi(i)=j)} = \frac{\sum_{\pi \in A_{i, j}}\m(T_j(\pi))}{\sum_{\pi \in A_{i, j}}\m(\pi)}, \label{eq:L2c}
\end{equation}
and (\ref{eq:L2a}) follows from (\ref{eq:L2b}) and (\ref{eq:L2c}). For the case $q>1$, the proof is similar. The lemma clearly also holds when $q=1$, which corresponds to the uniform measure on $S_n$.

\end{proof}


The following result establishes some bounds on the probability of a point in a Mallows permutation being within an interval.

\begin{lemma}\label{L3}
Suppose that $\{q_n\}_{n=1}^{\infty}$ is a sequence such that
\[
\lim_{n\to\infty} n (1-q_n) = \beta \in \mathbb{R}.
\]
For any sequence $\{a_n\}$ with $a_n \in [n]$ and any $0 \le y_1 < y_2 \le 1$,
\begin{align}
\limsup_{n \to \infty} \n \left( \frac{\pi(a_n)}{n} \in [y_1, y_2] \right) &\le (y_2 - y_1) e^{|\beta|}, \label{eq:L3a}\\
\liminf_{n \to \infty} \n \left( \frac{\pi(a_n)}{n} \in (y_1, y_2) \right) &\ge (y_2 - y_1) e^{-|\beta|}. \label{eq:L3b}
\end{align}
\end{lemma}

\begin{proof}
Here we only prove the case $\beta \ge 0$. The case $\beta < 0$ follows from the same argument. We also assume that $y_2 - y_1 < 1$, since the case $y_0 = 0, y_1 = 1$ can be verified easily.
Since $\lim_{n \to \infty} n (1-q_n) = \beta$ and $\lim_{n \to \infty} \frac{n \log{q_n}}{n(1 - q_n)} = -1$, we have
\[
\lim_{n \to \infty} q_n^n = \lim_{n \to \infty} e^{n \log{q_n}} = e^{-\beta}.
\]
Thus, for any $\delta >1 $, there exists $N>0$ such that $q_n^n \in \left( \frac{e^{-\beta}}{\delta}, \delta e^{-\beta}\right) $, when $n > N$.
By Lemma \ref{L2}, for any $n > N$ and any $i, s, t \in [n]$
\begin{equation}
\frac{\n (\pi(i) = s)} {\n (\pi(i) = t)} \le \max \left( q_n^n , \frac{1}{q_n^n} \right) < \delta e^{\beta}. \label{eq:L3c}
\end{equation}
Let $d = y_2 - y_1$ and $p_n = \min_{\{t: \frac{t}{n} \notin [y_1, y_2] \}} \left( \n (\pi(a_n) = t) \right)$. Note that the set $\{t: \frac{t}{n} \notin [y_1, y_2] \}$ is nonempty for sufficiently large $n$.
Then, by (\ref{eq:L3c}) and the fact that,
\[
\textstyle\Big|\Big\{k \in [n] :\frac{k}{n} \in [y_1, y_2]\Big\}\Big| \le nd+1, \quad \Big|\Big\{k \in [n]:\frac{k}{n} \notin [y_1, y_2]\Big\}\Big| \ge n(1-d)-1,
\]
we have,
\begin{align*}
\textstyle\n \left(\frac{\pi(a_n)}{n} \in [y_1, y_2] \right) &< (nd+1)\delta e^{\beta} p_n,\\
\textstyle\n \left(\frac{\pi(a_n)}{n} \notin [y_1, y_2] \right) &\ge (n(1-d) -1) p_n.
\end{align*}
Hence,
\begin{align*}
\n \Big(\textstyle\frac{\pi(a_n)}{n} \in [y_1, y_2] \Big) &< \frac{(nd+1)\delta e^{\beta}} {(n(1-d) -1)+(nd+1)\delta e^{\beta}}\\
 &< \frac{(nd+1)\delta e^{\beta}} {(n(1-d) -1)+(nd+1)}\\
 &= \frac{(nd+1)\delta e^{\beta}} {n},
\end{align*}
and (\ref{eq:L3a}) follows since $\delta$ can be chosen arbitrarily close to 1.
Similarly, to prove (\ref{eq:L3b}), define $p'_n = \min_{\{t: \frac{t}{n} \in (y_1, y_2) \}} \left( \n (\pi(a_n) = t) \right)$.
Then, by (\ref{eq:L3c}) and the fact that,
\[
\textstyle\Big|\Big\{k \in [n] :\frac{k}{n} \in (y_1, y_2)\Big\}\Big| \ge nd-1, \quad \Big|\Big\{k \in [n]:\frac{k}{n} \notin (y_1, y_2)\Big\}\Big| \le n(1-d)+1,
\]
we have
\begin{align*}
\textstyle\n \left(\frac{\pi(a_n)}{n} \in (y_1, y_2) \right) &\ge (nd-1)p'_n,\\
\textstyle\n \left(\frac{\pi(a_n)}{n} \notin (y_1, y_2) \right) &< (n(1-d)+1)\delta e^{\beta} p'_n.
\end{align*}
Hence,
\begin{align*}
\textstyle\n \left(\frac{\pi(a_n)}{n} \in (y_1, y_2) \right) &> \frac{(nd-1) }{(n(1-d)+1)\delta e^{\beta} + (nd-1)}\\
 &> \frac{(nd-1) }{(n(1-d)+1)\delta e^{\beta} + (nd-1)\delta e^{\beta}}\\
 &= \frac{(nd-1) e^{-\beta}}{n\delta}.
\end{align*}
(\ref{eq:L3b}) follows since $\delta$ can be chosen arbitrarily close to 1.\\
\end{proof}


In the next two lemmas, we introduce some properties of the density function $u(x, y, \beta)$ defined in Theorem \ref{T1}.

\begin{lemma}\label{L9} 
With $u(x, y, \beta)$ defined as in (\ref{eq:t1}), we have
\[
\int_{0}^{1} u(x, y, \beta) \, dx = 1, \qquad \forall y \in [0, 1],
\]
\[
\int_{0}^{1} u(x, y, \beta) \, dy = 1, \qquad \forall x \in [0, 1].
\]
\end{lemma}
\begin{proof} 
Since $\cosh{(x)}$ is an even function, $u(x, y, \beta)$ is symmetric with respect to the line $y = x$. That is
\[
u(x, y, \beta) = u(y, x, \beta), \qquad \forall x, y \in [0, 1].
\]
Hence we only need to show the first identity. By Corollary 6.2 in \cite{Starr},
\begin{equation}
\frac{\partial^2 \ln{u(x, y, \beta)}}{\partial x \partial y}  = 2 \beta u(x, y, \beta). \label{eq:L9a}
\end{equation}
Therefore, we have
\begin{equation}
\int_{0}^{1} u(x, y, \beta) \, dx =
\frac{1}{2\beta} \left( \frac{\partial \ln{u(1, y, \beta)}}{\partial y} - \frac{\partial \ln{u(0, y, \beta)}}{\partial y} \right).  \label{eq:L9b}
\end{equation}
Next, by direct calculation, we have
\begin{equation}
u(1, y, \beta) = \frac{(\beta/2) \sinh(\beta/2)}{\big( \frac{1}{2}e^{-\frac{\beta}{4}}(e^{\beta}-1)e^{-\frac{\beta y}{2}}\big)^2}=
\frac{\beta e^{\beta y}}{e^{\beta} - 1},  \label{eq:L9c}
\end{equation}
\begin{equation}
u(0, y, \beta) = \frac{(\beta/2) \sinh(\beta/2)}{\big( \frac{1}{2}e^{\frac{\beta}{4}}(1-e^{-\beta})e^{\frac{\beta y}{2}}\big)^2}=
\frac{\beta e^{-\beta y}}{1 - e^{-\beta}}. \label{eq:L9d}
\end{equation}
Hence, we get
\[
\frac{\partial \ln{u(1, y, \beta)}}{\partial y} = \beta \qquad \text{and} \qquad \frac{\partial \ln{u(0, y, \beta)}}{\partial y} = - \beta.
\]
By (\ref{eq:L9b}), the lemma follows.
\end{proof}

In the remainder of this section, we will simply use $u(x, y)$ to denote $u(x, y, \beta)$.


\begin{lemma}\label{L6}
For any $0 \le a, c, d \le 1$,
\[
-\beta \int_{c}^{d} \left(-\int_{0}^{a} u(x, y) \, dx + \int_{a}^{1}u(x, y) \, dx \right) \, dy = \ln{\frac{u(a, d)}{u(a, c)}}.
\]
\end{lemma}

\begin{proof}
For fixed $c, d \in [0, 1]$, define $f(a)$ to be the left-hand side of the identity and $g(a)$ to be the right-hand side of the identity.
Then, by Lemma \ref{L9} and (\ref{eq:L9d}), we have
\[
f(0) = g(0) = \beta (c - d).
\]
Hence, to prove the identity it suffices to show $f'(a) = g'(a)$ for any $a \in (0, 1)$.
Since $u(x, y)$ is bounded on $[0, 1]\times[0,1]$, we can change the order of integral and differentiation in the following,
\begin{align*}
f'(a) &= -\beta \frac{\partial}{\partial a}  \int_{c}^{d} \left(-\int_{0}^{a} u(x, y) \, dx + \int_{a}^{1}u(x, y) \, dx \right) \, dy\\
      &= -\beta \int_{c}^{d} \frac{\partial}{\partial a} \left(-\int_{0}^{a} u(x, y) \, dx + \int_{a}^{1}u(x, y) \, dx \right) \, dy\\
      &= -\beta \int_{c}^{d} \left( - u(a,y) - u(a, y) \right)\, dy\\
      &= 2 \beta \int_{c}^{d} u(a,y)\, dy.
\end{align*}
By (\ref{eq:L9a}), $\frac{\partial \ln{u(x, y)}}{\partial x}$ is the anti-derivative of $2\beta u(x, y)$ with respect to $y$. Thus we have
\[
g'(a) = \frac{\partial \ln{u(a, d)}}{\partial a} - \frac{\partial \ln{u(a, c)}}{\partial a} = 2 \beta \int_{c}^{d} u(a,y)\, dy.
\]

\end{proof}


\begin{lemma}\label{L7}
In the context of Lemma \ref{L3}, suppose $\{a_n\}_{n \ge 1}$ is a sequence such that $\lim_{n \to \infty} \frac{a_n}{n} = a$, where $a_n \in [n]$. 
For any $0 \le y_1 < y_2 < 1$,
\begin{equation}
\lim_{\epsilon \to 0^+} \limsup_{n \to \infty}
\left| \frac{\n \left( \frac{\pi(a_n)}{n} \in (y_2, y_2 + \epsilon) \right)}{\n \left( \frac{\pi(a_n)}{n} \in (y_1, y_1 + \epsilon) \right)}
- \frac{u(a, y_2)}{u(a, y_1)} \right| = 0.     \label{eq:L7a}
\end{equation}
For any $0 < y_1 < y_2 \le 1$,
\begin{equation}
\lim_{\epsilon \to 0^+} \limsup_{n \to \infty}
\left| \frac{\n \left( \frac{\pi(a_n)}{n} \in (y_2 - \epsilon, y_2) \right)}{\n \left( \frac{\pi(a_n)}{n} \in (y_1 - \epsilon, y_1) \right)}
- \frac{u(a, y_2)}{u(a, y_1)} \right| = 0.     \label{eq:L7b}
\end{equation}
\end{lemma}

\begin{proof}
Here we only prove (\ref{eq:L7a}), since (\ref{eq:L7b}) follows from the similar argument.\\
To prove (\ref{eq:L7a}), we need to show that for any $\eta >0$, there exists $\epsilon_0 >0$ such that for any fixed $\epsilon < \epsilon_0$, there exists $N>0$, which may depend on $\epsilon$, such that for any $n > N$, we have
\begin{equation}
\left| \frac{\n \left( \frac{\pi(a_n)}{n} \in (y_2, y_2 + \epsilon) \right)}{\n \left( \frac{\pi(a_n)}{n} \in (y_1, y_1 + \epsilon) \right)}
- \frac{u(a, y_2)}{u(a, y_1)} \right| < \eta  .  \label{eq:L7c}
\end{equation}
First, we define the following two rectangles:
\[
R_0 \coloneqq [0, a] \times [y_1, y_2], \qquad R_1 \coloneqq [a, 1] \times [y_1, y_2].\\
\]
Next define
\[
G(n, \lambda) \coloneqq \left\{ \pi \in S_n : \left| \LL(A) - u(A) \right| < \lambda, \text{ for any } A \in \{R_0, R_1\} \right\}.
\]
Let $\overline{G}(n, \lambda) \coloneqq S_n \setminus G(n, \lambda)$ denote the complement of $G(n, \lambda)$. Then,
\[
\overline{G}(n, \lambda) = \cup_{A \in \{R_0, R_1\}}
\left\{ \pi \in S_n : \left| \LL(A) - u(A) \right| \ge \lambda \right\}.
\]
Thus by Lemma \ref{L1}, for any $\epsilon_0 > 0$ and any $\lambda > 0$, we have
\begin{equation}
\lim_{n \to \infty} \n \left( \overline{G}(n, \lambda) \right) = 0 . \label{eq:L7d}
\end{equation}
Define
\[
GD(n, \lambda) \coloneqq \left\{ \pi \in S_n : Q(\pi, a_n) \cap G(n, \lambda/2) \neq \varnothing \right\}.
\]
Note that, for any rectangle $R$ and any $\tau, \xi \in Q(\pi, a_n)$,
\[
|L_{\tau}(R) - L_{\xi}(R)| \le \frac{1}{n}.
\]
Thus, when $n > \frac{2}{\lambda}$, it follows from triangle inequality that
\begin{equation}
GD(n, \lambda) \subset G(n, \lambda).  \label{eq:L7e}
\end{equation}
On the other hand, by the definition of $GD(n, \lambda)$ and the fact that, for any $i \in [n]$,  $\pi \in Q(\pi, i)$, it follows that
\begin{equation}
G(n, \lambda/2) \subset GD(n, \lambda). \label{eq:L7z}
\end{equation}
Hence by (\ref{eq:L7d}) and (\ref{eq:L7z}), for any $\lambda > 0$, we have
\begin{equation}
\lim_{n \to \infty} \n \left(GD(n, \lambda) \right) = 1 .  \label{eq:L7f}
\end{equation}
Next, given $\epsilon \in (0,  \epsilon_0)$ where the value of $\epsilon_0$ is to be determined, define
\[
\textstyle A_n \coloneqq \{ \pi \in S_n :  \frac{\pi(a_n)}{n} \in (y_1, y_1 + \epsilon) \},
\quad
B_n \coloneqq \{ \pi \in S_n :  \frac{\pi(a_n)}{n} \in (y_2, y_2 + \epsilon) \}.
\]
Then, by Lemma \ref{L3}, when $n$ is sufficiently large, we have
\[
\textstyle\n (A_n) > \frac{\epsilon}{2}\,e^{-|\beta|}, \qquad \n (B_n) > \frac{\epsilon}{2} \, e^{-|\beta|}.
\]
Thus, by (\ref{eq:L7f}), there exists an $N_1 > 0$ such that, for any $n > N_1$, we have
\[
 \left| \frac{\n \left(B_n \cap GD(n, \lambda) \right)}{\n \left(A_n \cap GD(n, \lambda) \right)}
  - \frac{\n (B_n)}{\n (A_n)} \right| < \frac{\eta}{2}.
\]
Therefore, to prove (\ref{eq:L7c}), it suffices to show that for sufficiently large $n$, we have
\begin{equation}
\left| \frac{\n \left(B_n \cap GD(n, \lambda) \right)}{\n \left(A_n \cap GD(n, \lambda) \right)}
- \frac{u(a, y_2)}{u(a, y_1)} \right| < \frac{\eta}{2}.  \label{eq:L7g}
\end{equation}
In order to prove (\ref{eq:L7g}), we are going to exploit two things. The first one is the fact that $\{ Q(\pi, a_n) : \pi \in GD(n, \lambda) \}$ is a partition of $GD(n, \lambda)$. The second is the following,
\[
\frac{c_i}{d_i} > r,\, c_i>0,\, d_i>0 \text{ for } \forall i \in [m]\, \Rightarrow\, \frac{\sum_{i=1}^{m}c_i}{\sum_{i=1}^{m}d_i} > r,
\]
\[
\frac{c_i}{d_i} < r,\, c_i>0,\, d_i>0 \text{ for } \forall i \in [m]\, \Rightarrow\, \frac{\sum_{i=1}^{m}c_i}{\sum_{i=1}^{m}d_i} < r.
\]
Hence, to prove (\ref{eq:L7g}), it suffices to show that, for sufficiently large $n$, we have
\begin{equation}
\left| \frac{\n \left(B_n \cap Q(\pi, a_n) \right)}{\n \left(A_n \cap Q(\pi, a_n) \right)}
- \frac{u(a, y_2)}{u(a, y_1)} \right| < \frac{\eta}{2} ,  \label{eq:L7h}
\end{equation}
for any $Q(\pi, a_n) \subset GD(n, \lambda)$. Note that $A_n \cap Q(\pi, a_n)$ is nonempty for any $\pi \in S_n$, when $n > 1/\epsilon$.
The strategy to prove (\ref{eq:L7h}) is the following, we show that when $n$ is sufficiently large, for any $Q(\pi, a_n) \subset GD(n, \lambda)$ and any $\tau \in B_n \cap Q(\pi, a_n)$, $\xi \in A_n \cap Q(\pi, a_n)$, we have
\begin{equation}
\Big| \frac{1}{n}\left( \inv(\tau) - \inv(\xi) \right) - I\,\Big| < 2 \lambda + 4 \epsilon + \frac{4}{n}. \label{eq:L7k}
\end{equation}
Here 
\[
I \coloneqq \int_{y_1}^{y_2} \left(-\int_{0}^{a} u(x, y) \, dx + \int_{a}^{1}u(x, y) \, dx \right) \, dy = u(R_1) - u(R_0).
\]
Note that $\frac{\n(\tau)}{\n(\xi)} = q_n^{\inv(\tau) - \inv(\xi)}$. Thus, by (\ref{eq:L7k}), for any $\tau \in B_n \cap Q(\pi, a_n)$, $\xi \in A_n \cap Q(\pi, a_n)$, we have
\[
q_n^{n(I + 2 \lambda + 4 \epsilon + 4/n)} \le \frac{\n(\tau)}{\n(\xi)} \le q_n^{n(I - 2 \lambda - 4 \epsilon - 4/n)}.
\]
Here we assume $0 < q_n < 1$. (The cases $q_n > 1$ and $q_n = 1$ follow by similar argument.) By the definition of $A_n, B_n$, we have 
\[
n\epsilon - 1 \le |A_n \cap Q(\pi, a_n)|,\ |B_n \cap Q(\pi, a_n)| \le n\epsilon + 1.
\]
Hence we have
\[
\frac{n\epsilon -1}{n\epsilon + 1} q_n^{n(I + 2 \lambda + 4 \epsilon + 4/n)}
\le \frac{\n \left(B_n \cap Q(\pi, a_n) \right)}{\n \left(A_n \cap Q(\pi, a_n) \right)}
\le \frac{n\epsilon + 1}{n\epsilon - 1} q_n^{n(I - 2 \lambda - 4 \epsilon - 4/n)}.
\]
By Lemma \ref{L6} and the fact that $\lim_{n \to \infty} q_n^n = e^{-\beta}$ and $\lim_{n \to \infty} q_n = 1$ , we have
\[
\lim_{n \to \infty} \frac{n\epsilon -1}{n\epsilon + 1} q_n^{n(I + 2 \lambda + 4 \epsilon + 4/n)} =
\frac{u(a, y_2)}{u(a, y_1)} e^{-\beta (2 \lambda + 4 \epsilon)},
\]
\[
\lim_{n \to \infty} \frac{n\epsilon + 1}{n\epsilon - 1} q_n^{n(I - 2 \lambda - 4 \epsilon - 4/n)} =
\frac{u(a, y_2)}{u(a, y_1)} e^{\beta (2 \lambda + 4 \epsilon)}.
\]
Thus, we can choose $\epsilon_0$ and $\lambda$ small enough such that, for any $\epsilon \in (0, \epsilon_0)$, (\ref{eq:L7h}) holds for sufficiently large $n$.\\

The remaining part of the proof is to show (\ref{eq:L7k}). Suppose $n$ is sufficiently large such that $\frac{a_n}{n} \in (a - \epsilon, a + \epsilon)$. Without loss of generality, suppose
$\frac{a_n}{n} \in [a, a + \epsilon)$. (The other case can be shown in a similar argument.) By Proposition \ref{P1}, for any $Q(\pi, a_n) \subset GD(n, \lambda)$, and for any $\tau \in B_n \cap Q(\pi, a_n)$, $\xi \in A_n \cap Q(\pi, a_n)$, we have
\begin{align*}
  &\inv(\tau) - \inv(\xi)\\
 =\,&\textstyle|\{t > a_n : \xi(a_n) < \xi(t) \le \tau(a_n)\}| - |\{t < a_n : \xi(a_n) < \xi(t) \le \tau(a_n)\}|\\
 =\,&\textstyle|\{\frac{t}{n} > \frac{a_n}{n} : \frac{\xi(a_n)}{n} < \frac{\xi(t)}{n} \le \frac{\tau(a_n)}{n} \}|
    - |\{\frac{t}{n} < \frac{a_n}{n} : \frac{\xi(a_n)}{n} < \frac{\xi(t)}{n} \le \frac{\tau(a_n)}{n} \}|\\
 \le\,&\textstyle|\{\frac{t}{n} > a : y_1 < \frac{\xi(t)}{n} < y_2 + \epsilon \}|
    - |\{\frac{t}{n} < a : y_1 + \epsilon \le \frac{\xi(t)}{n} \le y_2 \}|\\
 =\, &\textstyle|\{t : \big(\frac{t}{n}, \frac{\xi(t)}{n}\big) \in (a, 1] \times (y_1, y_2 + \epsilon) \}|\\
      &\textstyle\qquad\qquad\qquad\qquad\qquad - |\{t : \big(\frac{t}{n}, \frac{\xi(t)}{n}\big) \in (0, a) \times [y_1 + \epsilon, y_2 ] \}|\\
 \le\,&\textstyle|\{t : \big(\frac{t}{n}, \frac{\xi(t)}{n}\big) \in (a, 1] \times (y_1, y_2] \}| + (n\epsilon+1)\\
     &\textstyle\qquad\qquad\qquad\qquad\qquad - |\{t : \big(\frac{t}{n}, \frac{\xi(t)}{n}\big) \in (0, a) \times [y_1 , y_2] \}| + (n\epsilon+1)\\
 \le\,&\textstyle|\{t : \big(\frac{t}{n}, \frac{\xi(t)}{n}\big) \in [a, 1] \times [y_1, y_2] \}|\\
     &\textstyle\qquad\qquad\qquad\qquad\qquad - |\{t : \big(\frac{t}{n}, \frac{\xi(t)}{n}\big) \in [0, a] \times [y_1 , y_2] \}| + 2n\epsilon +4\\
 =\, &\textstyle n\lx\big([a, 1] \times [y_1, y_2] \big) - n\lx\big([0, a] \times [y_1, y_2] \big) + 2n\epsilon +4\\
 =\, &\textstyle n\lx(R_1 ) - n\lx( R_0 ) + 2n\epsilon +4.
\end{align*}
The first inequality above follows because $\frac{a_n}{n} \ge a$, $\frac{\xi(a_n)}{n} \in (y_1, y_1 + \epsilon)$ and $\frac{\tau(a_n)}{n} \in (y_2, y_2 + \epsilon)$.
The second inequality follows because
\[
\textstyle|\{t \in [n] : \frac{\xi(t)}{n} \in (y_2, y_2+\epsilon) \}| \le n\epsilon +1,
\]
\[
\textstyle|\{t \in [n] : \frac{\xi(t)}{n} \in [y_1, y_1+\epsilon) \}| \le n\epsilon +1.
\]
The third inequality follows because, since we change $(0, a)$ to $[0, a]$ in the second term, we add two in the end to compensate the possible
extra subtraction. Hence, we have
\begin{align}
\textstyle\frac{1}{n}(\inv(\tau) - \inv(\xi))
&\textstyle\le \lx(R_1 ) - \lx(R_0) + 2\epsilon + \frac{4}{n}   \label{eq:L7U}\\
&\textstyle\le u(R_1) - u(R_0) + 2\lambda + 2\epsilon + \frac{4}{n}  \nonumber\\
&\textstyle = I + 2\lambda + 2\epsilon + \frac{4}{n}. \nonumber
\end{align}
Here we use the fact that, by (\ref{eq:L7e}), $\xi \in GD(n, \lambda) \subset G(n, \lambda)$.

Similarly, to show the lower bound of $\inv(\tau) - \inv(\xi)$, we have
\begin{align*}
  &\inv(\tau) - \inv(\xi)\\
 =\,&\textstyle|\{t > a_n : \xi(a_n) < \xi(t) \le \tau(a_n)\}| - |\{t < a_n : \xi(a_n) < \xi(t) \le \tau(a_n)\}|\\
 =\,&\textstyle|\{\frac{t}{n} > \frac{a_n}{n} : \frac{\xi(a_n)}{n} < \frac{\xi(t)}{n} \le \frac{\tau(a_n)}{n} \}|
    - |\{\frac{t}{n} < \frac{a_n}{n} : \frac{\xi(a_n)}{n} < \frac{\xi(t)}{n} \le \frac{\tau(a_n)}{n} \}|\\
 =\, &\textstyle|\{\frac{t}{n} > a : \frac{\xi(a_n)}{n} < \frac{\xi(t)}{n} \le \frac{\tau(a_n)}{n} \}|
    - |\{ \frac{a_n}{n} \ge \frac{t}{n} > a : \frac{\xi(a_n)}{n} < \frac{\xi(t)}{n} \le \frac{\tau(a_n)}{n} \}|\\
 &\textstyle - |\{\frac{t}{n} < a : \frac{\xi(a_n)}{n} < \frac{\xi(t)}{n} \le \frac{\tau(a_n)}{n} \}|
     - |\{a \le \frac{t}{n} < \frac{a_n}{n} : \frac{\xi(a_n)}{n} < \frac{\xi(t)}{n} \le \frac{\tau(a_n)}{n} \}|\\
 \ge\,&\textstyle|\{\frac{t}{n} > a : y_1 + \epsilon \le \frac{\xi(t)}{n} \le y_2 \}| - (n \epsilon +1)\\
        &\textstyle\qquad\qquad\qquad\qquad\qquad\qquad - |\{\frac{t}{n} < a : y_1 < \frac{\xi(t)}{n} < y_2 + \epsilon \}|- (n \epsilon +1)\\
 =\, &\textstyle|\{t : \big(\frac{t}{n}, \frac{\xi(t)}{n}\big) \in (a, 1] \times [y_1 + \epsilon, y_2]\}|\\
     &\textstyle\qquad\qquad\qquad\qquad - |\{t : \big(\frac{t}{n}, \frac{\xi(t)}{n}\big) \in (0, a) \times (y_1, y_2 + \epsilon) \}|
       - 2(n \epsilon +1)\\
 \ge\,&\textstyle|\{t : \big(\frac{t}{n}, \frac{\xi(t)}{n}\big) \in (a, 1] \times [y_1, y_2]\}| - (n\epsilon+1)\\
     &\textstyle\qquad\qquad\quad - |\{t : \big(\frac{t}{n}, \frac{\xi(t)}{n}\big) \in (0, a) \times (y_1, y_2] \}| - (n\epsilon+1)
       - 2(n \epsilon +1)\\
 =\, &\textstyle n\lx\big([a, 1] \times[y_1, y_2] \big) - n\lx\big([0, a] \times[y_1, y_2]\big) - 4n\epsilon -4\\
 =\, &\textstyle n\lx(R_1 ) - n\lx( R_0 ) - 4n\epsilon -4.
\end{align*}
The first inequality above follows since, by the definition of $A_n, B_n$, we have $\frac{\xi(a_n)}{n} \in (y_1, y_1 + \epsilon)$, $\frac{\tau(a_n)}{n} \in (y_2, y_2 + \epsilon)$
and, since $\frac{a_n}{n}\in [a, a + \epsilon)$,
\[
\textstyle|\{t\in [n]: \frac{a_n}{n} \ge \frac{t}{n} > a\}| \le n\epsilon +1,
\qquad |\{t\in[n]: a \le \frac{t}{n} < \frac{a_n}{n}\}| \le n\epsilon +1.
\]
The second inequality follows because
\[
\textstyle|\{t \in [n] : \frac{\xi(t)}{n} \in [y_1, y_1+\epsilon)  \}| \le n\epsilon +1,
\]
\[
\textstyle|\{t \in [n] : \frac{\xi(t)}{n} \in (y_2, y_2+\epsilon) \}| \le n\epsilon +1.
\]
Hence, we have
\begin{align}
\textstyle\frac{1}{n}(\inv(\tau) - \inv(\xi))
&\textstyle\ge \lx(R_1 ) - \lx(R_0) - 4\epsilon - \frac{4}{n}   \label{eq:L7L}\\
&\textstyle\ge u(R_1) - u(R_0) - 2\lambda - 4\epsilon - \frac{4}{n}  \nonumber\\
&\textstyle= I - 2\lambda - 4\epsilon - \frac{4}{n}. \nonumber
\end{align}
Here again we use the fact that, by (\ref{eq:L7e}), $\xi \in GD(n, \lambda) \subset G(n, \lambda)$.
The fact that (\ref{eq:L7k}) follows from (\ref{eq:L7U}) and (\ref{eq:L7L}) completes the proof.

\end{proof}

To complete the proof of Lemma \ref{M1} we use the following result (cf. 7.2.5 in \cite{Ash}) and the next two lemmas.

\begin{theorem} \label{thm:AshDade}
Let $\{u_n\}_{n \ge 1}$ be a sequence of finite measure on $\mathbb{R}$. If\, $\{u_n\}_{n \ge 1}$ is tight, and every weakly convergent
subsequence of \,$\{u_n\}_{n \ge 1}$ converges to the measure $v$, then $u_n \overset{d}{\longrightarrow} v$.
\end{theorem}


\begin{lemma}\label{L4}
In the context of Lemma \ref{M1}, let
$\{a_{t_n}\}$ be a subsequence of $\{a_n\}$ such that
\[
\mu_{t_n, q_{t_n}} \left( \frac{\pi(a_{t_n})}{t_n} \in ( \cdot ) \right) \overset{d}{\longrightarrow} v.
\]
Then the distribution function $F_v(y)$ of the limit probability measure $v$ is absolutely continuous. Here
$\mu_{t_n, q_{t_n}} \left( \frac{\pi(a_{t_n})}{t_n} \in ( \cdot ) \right)$ denotes the probability measure induced by
$\frac{\pi(a_{t_n})}{t_n}$ under $\mu_{t_n, q_{t_n}}$.
\end{lemma}

\begin{proof}
For any $\epsilon > 0$, let $\delta = \frac{\epsilon}{4 e^{|\beta|}}$. By the definition of absolute continuity, we will show that, for any
$\{(y_1, y_2), (y_3, y_4), \dots, (y_{2m-1}, y_{2m})\}$ with $y_{2k-1} < y_{2k}$ and
$\sum_{k=1}^{m} |y_{2k} - y_{2k-1}| < \delta$, we have $\sum_{k=1}^{m} |F_v(y_{2k}) - F_v(y_{2k-1})| < \epsilon $.
Without loss of generality, we may assume that every $y_i$ is a continuous point of $F_v(y)$ with $0 \le y_i \le 1$. Since there are at most countably many discontinuity of $F_v(y)$, we can always choose a new set of interval $\{(y'_{2k-1}, y'_{2k})\}$  such that $F_v(y)$ is continuous at every $y'_i$,
$ [y_{2k-1}, y_{2k}] \subset [y'_{2k-1}, y'_{2k}]$ and $\sum_{k=1}^{m} |y'_{2k} - y'_{2k-1}| < \delta$ still holds. Next, for the simplicity of notation, define
\begin{equation}
v_n \coloneqq \mu_{t_n, q_{t_n}} \left( \frac{\pi(a_{t_n})}{t_n} \in ( \cdot ) \right). \label{eq:L4a}
\end{equation}
By Lemma \ref{L3}, there exists $N_1>0$ such that for any $n>N_1$,
\[
v_n \left( [y_{2k-1}, y_{2k}] \right) \le 2 (y_{2k}- y_{2k-1}) e^{|\beta|},
\]
for all $k \in [m]$. Since $v_n \overset{d}{\longrightarrow} v$, there exists
$N_2 >0$ such that for any $n>N_2$,
\[
\left|F_v(y_{2k}) - F_v(y_{2k-1}) - v_n \left( [y_{2k-1}, y_{2k}] \right) \right| < \frac{\epsilon}{2m},
\]
for all $k \in [m]$. Let $n = \max{(N_1, N_2)}+1$, we have
\begin{align*}
 &\sum_{k=1}^{m} |F_v(y_{2k}) - F_v(y_{2k-1})|\\
 \le& \sum_{k=1}^{m}\left|F_v(y_{2k}) - F_v(y_{2k-1})- v_n \left( [y_{2k-1}, y_{2k}] \right) \right| + \sum_{k=1}^{m} v_n \left( [y_{2k-1}, y_{2k}] \right)\\
 <&\, \frac{\epsilon}{2} + 2 e^{|\beta|} \sum_{k=1}^{m}(y_{2k} - y_{2k-1})\\
 <&\, \frac{\epsilon}{2} + 2 e^{|\beta|} \delta\\
 =&\  \epsilon.
\end{align*}
\end{proof}


\begin{lemma}\label{L5}
In the context of Lemma \ref{L4}, we have
\[
F_v(y) = \int_{0}^{y} u(a, t, \beta)\, dt,
\]
for any $y \in [0, 1]$. Here $u(x, y, \beta)$ is defined in (\ref{eq:t1}).
\end{lemma}
\begin{proof}
For the simplicity of notation, we will use $u(x, y)$ to denote $u(x, y, \beta)$. By Lemma \ref{L4}, $F_v(y)$ is absolutely continuous. Hence $F_v(y)$ is differentiable almost
everywhere, say $F'_v(y) = f(y)$ a.e.\ on $[0, 1]$, and moreover, we have $F_v(y) = \int_{0}^{y} f(t) \, dt$. Here we use the fact that the support of $v$ is $[0, 1]$.
Note that, by Lemma \ref{L3}, for any $y \in (0, 1)$ such that $F'_v(y) = f(y)$, we have $f(y) \ge e^{-|\beta|} > 0$.
Then in order to show $f(y) = u(a, y)$ a.e., it suffices to show
\begin{equation}
\frac{f(y_2)}{f(y_1)} = \frac{u(a, y_2)}{u(a, y_1)}, \label{eq:L5a}
\end{equation}
for any $y_1, y_2 \in A$, where $A \coloneqq \{ y\in (0,1): F'_v(y) = f(y)\}$. This is because, for any $y \in A$, we have
\[
\textstyle\frac{1}{f(y)} = \int_{0}^{1} \frac{f(z)}{f(y)} \, dz = \int_{A} \frac{f(z)}{f(y)} \, dz
= \int_{A} \frac{u(a,z)}{u(a,y)} \, dz = \int_{0}^{1} \frac{u(a,z)}{u(a,y)} \, dz = \frac{1}{u(a, y)}.
\]
Here we use the fact that the Lebesgue measure of $A$ is 1 as well as Lemma \ref{L9} in the last equality.
Next, since we have
\begin{align*}
\lim_{\epsilon \to 0^+} \frac{v((y_2, y_2+\epsilon))}{v((y_1, y_1+\epsilon))}
=& \lim_{\epsilon \to 0^+} \frac{F_v(y_2+\epsilon) - F_v(y_2)}{F_v(y_1+\epsilon) - F_v(y_1)}\\
=& \lim_{\epsilon \to 0^+} \frac{F_v(y_2+\epsilon) - F_v(y_2)}{\epsilon}\Big/\frac{F_v(y_1+\epsilon) - F_v(y_1)}{\epsilon}\\
=&\, \frac{f(y_2)}{f(y_1)}.
\end{align*}
Thus, to prove (\ref{eq:L5a}), it suffices to show that
\begin{equation}
\lim_{\epsilon \to 0^+} \left|\frac{v((y_2, y_2+\epsilon))}{v((y_1, y_1+\epsilon))} - \frac{u(a, y_2)}{u(a, y_1)}\right| = 0. \label{eq:L5b}
\end{equation}
Next, inheriting the notation in (\ref{eq:L4a}),  since $v_n \overset{d}{\longrightarrow} v$ and $F_v(y)$ is continuous, we have
\begin{equation}
\lim_{n \to \infty} \left| \frac{v_n((y_2, y_2+\epsilon))}{v_n((y_1, y_1+\epsilon))} - \frac{u(a, y_2)}{u(a, y_1)} \right|
= \left| \frac{v((y_2, y_2+\epsilon))}{v((y_1, y_1+\epsilon))} - \frac{u(a, y_2)}{u(a, y_1)} \right|. \label{eq:L5c}
\end{equation}
Since $\{v_n\}$ is a subsequence of $\textstyle\Big\{ \n \Big( \frac{\pi(a_n)}{n} \in (\cdot) \Big) \Big\}$, by Lemma~\ref{L7}, (\ref{eq:L5b}) follows from (\ref{eq:L5c}).
\end{proof}

\begin{proof}[Proof of Lemma \ref{M1}]
Since the support of $\n \Big( \frac{\pi(a_n)}{n} \in (\cdot) \Big)$ is within $[0, 1]$, the sequence
$\left\{ \n \Big( \frac{\pi(a_n)}{n} \in (\cdot) \Big) \right\}$ is tight.
Therefore, Lemma \ref{M1} follows from Lemma~\ref{L4}, Lemma~\ref{L5} and Theorem~\ref{thm:AshDade}.
\end{proof}

\subsection{Proof of Lemma \ref{L18} and Lemma \ref{L13}}\label{section5}
We now complete the proofs of Lemma \ref{L18} and Lemma \ref{L13}.
\begin{definition}
For any $\pi \in S_n$ and any $ 1 \le j < k \le n$, let $\pi([j, k])$ denote the vector $\left( \pi(j), \pi(j+1), \cdots, \pi(k) \right)$.
Let $\pi_{[j, k]}$ denote the permutation in $S_{k-j+1}$ induced by $\pi([j, k])$, i.\,e.
\[
\pi_{[j, k]}(i) = \sum_{s = j}^{k} \mathds{1}_{\{\pi(s) \le \pi(j+i-1) \}}, \qquad  \forall i \in [k-j+1].
\]
\end{definition}

We make use of the following property of the Mallows distribution (see e.g. Lemma 2.5 and Lemma 2.6 in \cite{Naya}):
\begin{proposition}\label{P2}
Given $\pi \sim \m$, for any $1 < k < n$, we have $\pi_{[1, k]} \sim \mu_{k, q}$, $\pi_{[k+1, n]} \sim \mu_{n-k, q}$ and $\pi_{[1, k]}, \pi_{[k+1, n]}$ are independent.
\end{proposition}


\begin{lemma}\label{L11}
For any $0 \le a < b \le 1$ and $y \in [0, 1]$, we have the following identity
\[
\int_{0}^{y} u(a, t, \beta)\,dt = \int_{0}^{y'} u\left(\frac{a}{b}, t, b \beta\right)\,dt, \quad \forall \beta \in \mathbb{R}.
\]
Here,
\[
y' \coloneqq \frac{1}{b}\,u_{\beta}([0, b]\times[0, y]) = \frac{1}{b} \int_{0}^{b}\int_{0}^{y} u(x, t, \beta)\, dtdx
\]
and $u(x, y, \beta)$ is defined in (\ref{eq:t1}).
\end{lemma}

We make some preparation before proving Lemma \ref{L11}.
Given $a, b \in [0, 1]$, choose two sequences $\{a_n\}$ and $\{b_n\}$ such that $a_n \in [n]$, $b_n \in [n]$ and
\[
\lim_{n \to \infty} \frac{a_n}{n} = a, \qquad \lim_{n \to \infty} \frac{b_n}{n} = b.
\]
Moreover, for any $\beta \in \mathbb{R}$, choose a sequence $\{q_n\}$ with $q_n >0$ such that
$\lim_{n \to \infty} n (1-q_n) = \beta$.
By Lemma \ref{M1}, we have
\begin{equation}
\lim_{n \to \infty} \textstyle\n \left( \frac{\pi(a_n)}{n} \le y \right) = \int_{0}^{y} u(a, t, \beta)\,dt. \label{eq:L11aa}
\end{equation}
We will show that
\begin{equation}
\lim_{n \to \infty} \textstyle\n \left( \frac{\pi(a_n)}{n} \le y \right) = \int_{0}^{y'} u\left(\frac{a}{b}, t, b \beta\right)\,dt. \label{eq:L11a}
\end{equation}
Lemma \ref{L11} follows from (\ref{eq:L11aa}) and (\ref{eq:L11a}). First, regarding $\{a_n\}$ and $\{b_n\}$ as fixed sequences, $y$ as a fixed number, we make the following definitions,
\[
R_0 \coloneqq [0, b] \times [0, y], \qquad R_1 \coloneqq [b, 1] \times [0, y],
\]
\[
K_n  \coloneqq \{ ( v_1, v_2, \cdots, v_{n - b_n +1} ): v_i \in [n] \text{ and } i \neq j \Rightarrow v_i \neq v_j \},
\]
\[
f_n(v) \coloneqq |\{ v_i \in v : v_i \le n y \} | \quad \text{ for } v \in K_n,
\]
\[
G_n(\lambda) \coloneqq \left\{ v \in K_n: \left| \frac{1}{n}f_n(v) - u_{\beta}(R_1) \right| < \lambda \right\}.
\]
Here $K_n$ consists of all possible values $\pi([b_n, n])$ can take when $\pi \in S_n$.  $f_n(\pi([b_n, n]))$ denotes the number of
points $\left( \frac{i}{n}, \frac{\pi(i)}{n} \right)$ inside the rectangle $[\frac{b_n}{n}, 1] \times [0, y]$.

Next we show that, for any $\lambda > 0 $,
\begin{equation}
\lim_{n \to \infty} \n \left( \pi([b_n, n]) \notin G_n(\lambda) \right) = 0 \label{eq:L11b}
\end{equation}

\begin{proof}[Proof of (\ref{eq:L11b})]
First, since the difference between $[\frac{b_n}{n}, 1] \times [0, y]$ and $R_1$ is a rectangle with width $\left| \frac{b_n}{n} - b \right|$, it follows that
\begin{align*}
  &\left| f_n(\pi([b_n, n])) - n \LL(R_1) \right|\\
 =& \textstyle\left| \big|\{ i : \big( \frac{i}{n}, \frac{\pi(i)}{n} \big) \in [\frac{b_n}{n}, 1] \times [0, y] \} \big| -
    \big| \{ i : \big( \frac{i}{n}, \frac{\pi(i)}{n} \big) \in R_1 \} \big|\right|\\
 \le&\, |b_n - n b| +1.
\end{align*}
Thus, for any $\lambda > 0$, there exists a $N >0$ such that for all $n > N$,
\[
\textstyle\left| \frac{1}{n} f_n(\pi([b_n, n])) - \LL(R_1) \right| \le \left| \frac{b_n}{n} - b \right| + \frac{1}{n} < \frac{\lambda}{2}.
\]
Here we use the fact that $\lim_{n \to \infty} \frac{b_n}{n} = b$.
Hence, for any $n > N$, we have
\begin{align*}
  & \left|\frac{1}{n} f_n(\pi([b_n, n])) - u_{\beta}(R_1) \right| \ge \lambda\\
 \Rightarrow & \left| \frac{1}{n} f_n(\pi([b_n, n])) - \LL(R_1) \right| + \left| \LL(R_1) - u_{\beta}(R_1) \right| \ge \lambda\\
 \Rightarrow & \left| \LL(R_1) - u_{\beta}(R_1) \right| > \frac{\lambda}{2}.
\end{align*}
Thus,
\begin{align*}
 & \, \n \left( \pi([b_n, n]) \notin G_n(\lambda) \right) \\
=& \, \n \left( \left|\frac{1}{n} f_n(\pi([b_n, n])) - u_{\beta}(R_1) \right| \ge \lambda \right) \\
\le& \, \n \left( \left| \LL(R_1) - u_{\beta}(R_1) \right| > \frac{\lambda}{2} \right).
\end{align*}
(\ref{eq:L11b}) follows from the above inequality and Lemma \ref{L1}.
\end{proof}

Next we show that, for any $\epsilon > 0$, we can choose a sufficiently small $\lambda$ and $N >0$ such that for all $ n > N $ and any
$v \in G_n(\lambda)$,
\begin{equation}
\textstyle\left|\, \n \left( \frac{\pi(a_n)}{n} \le y \ \middle|\  \pi([b_n, n]) = v \right) -
\int_{0}^{y'} u\left(\frac{a}{b}, t, b \beta\right)dt \,\right| < \frac{\epsilon}{3}. \label{eq:L11c}
\end{equation}

\begin{proof}[Proof of (\ref{eq:L11c})]
Assume $n$ is sufficiently large such that $a_n < b_n$. For any $v \in G_n(\lambda)$, here the value of $\lambda$ is to be determined, we have
\begin{align}
& \textstyle\n \left( \frac{\pi(a_n)}{n} \le y \ \middle|\  \pi([b_n, n]) = v \right) \label{eq:L11d}\\
 =\ & \n \left( \pi(a_n) \le n y \ \middle|\ \pi([b_n, n]) = v \right) \nonumber\\
 =\ & \n \left( \pi_{[1, b_n -1]} (a_n) \le \lfloor n y \rfloor - f_n(v) \ \middle|\ \pi([b_n, n]) = v \right) \nonumber\\
 =\ & \mu_{b_n -1, q_n} \left( \tau(a_n) \le \lfloor n y \rfloor - f_n(v) \right) \nonumber\\
 =\ & \mu_{b_n -1, q_n} \textstyle\left( \frac{\tau(a_n)}{b_n -1} \le \frac{1}{b_n -1} \left( \lfloor n y \rfloor - f_n(v) \right) \right) \nonumber\\
 =\ & \mu_{b_n -1, q_n} \textstyle\left( \frac{\tau(a_n)}{b_n -1} \le
  \frac{n}{b_n -1} \left( \frac{\lfloor n y \rfloor}{n} - \frac{f_n(v)}{n} \right) \right). \nonumber
\end{align}
The second equality follows since, conditioned on $\pi([b_n, n]) = v$, we have
\[
\{ \pi \in S_n : \pi(a_n) \le n y \}
= \{ \pi \in S_n : \pi_{[1, b_n -1]} (a_n) \le \lfloor n y \rfloor - f_n(v) \}.
\]
Note that $\lfloor n y \rfloor - f_n(v)$ is the number of $i \le n y$ which is not in $v$.
The third equality is due to Proposition \ref{P2} with $\tau \sim \mu_{b_n -1, q_n}$. Next, by the following facts,
\begin{equation}\label{eq:L11D}
\lim_{n \to \infty} (b_n -1)(1-q_n) = \lim_{n \to \infty}\textstyle\frac{b_n -1}{n}\cdot\lim_{n \to \infty}n (1-q_n) = b \beta,
\end{equation}
\vspace{-4mm}
\begin{equation}
\lim_{n \to \infty} \frac{a_n}{b_n-1} = \lim_{n \to \infty} \frac{a_n}{n}\cdot\lim_{n \to \infty}\frac{n}{b_n -1} = \frac{a}{b},
\end{equation}
\begin{equation}\label{eq:L11e}
\lim_{n \to \infty} \frac{n}{b_n -1} \left( \frac{\lfloor n y \rfloor}{n} - u_{\beta}(R_1) \right)
= \frac{1}{b}(y - u_{\beta}(R_1)) = \frac{1}{b}\, u_{\beta}(R_0) = y', 
\end{equation}
and Lemma \ref{M1}, we have
\[
\lim_{n \to \infty} \textstyle\mu_{b_n -1, q_n} \left( \frac{\tau(a_n)}{b_n -1} \le \frac{n}{b_n -1} \left( \frac{\lfloor n y \rfloor}{n} - u_{\beta}(R_1) \right) \right)
= \int_{0}^{y'} u\left(\frac{a}{b}, t, b \beta\right)dt.
\]
Hence, there exists $N_1 >0$ such that for any $n > N_1$,
\begin{equation}
\textstyle\left|\mu_{b_n -1, q_n} \left( \frac{\tau(a_n)}{b_n -1} \le \frac{n}{b_n -1} \left( \frac{\lfloor n y \rfloor}{n} - u_{\beta}(R_1) \right) \right) - \int_{0}^{y'} u\left(\frac{a}{b}, t, b \beta\right)dt \right| < \frac{\epsilon}{6}. \label{eq:L11e0}
\end{equation}
By (\ref{eq:L11e}), there exists $N_2 > 0$ such that for all $n > N_2$,
\begin{equation}
\textstyle\frac{n}{b_n -1} < \frac{2}{b} \quad
 \text{and} \quad \left| \frac{n}{b_n -1} \left( \frac{\lfloor n y \rfloor}{n} - u_{\beta}(R_1) \right) - y' \right| < \lambda. \label{eq:L11e1}
\end{equation}
Hence, for any $n > N_2$ and any $v \in G_n(\lambda)$, we have
\begin{align}
  & \textstyle\left| \frac{n}{b_n -1} \left( \frac{\lfloor n y \rfloor}{n} - \frac{f_n(v)}{n} \right) - y' \right| \label{eq:L11e2}\\
 \le\,& \textstyle\left| \frac{n}{b_n -1} \left( \frac{\lfloor n y \rfloor}{n} - \frac{f_n(v)}{n} \right) -
        \frac{n}{b_n -1} \left( \frac{\lfloor n y \rfloor}{n} - u_{\beta}(R_1) \right) \right| \nonumber\\
    & \qquad\qquad\qquad\qquad\qquad\qquad + \textstyle\left| \frac{n}{b_n -1} \left( \frac{\lfloor n y \rfloor}{n} - u_{\beta}(R_1) \right) - y' \right| \nonumber\\
 < & \textstyle\left(\frac{2}{b} + 1\right) \lambda. \nonumber
\end{align}
Let $C \coloneqq \frac{2}{b} + 1$. Since, by (\ref{eq:L11e1}) and (\ref{eq:L11e2}), both $\frac{n}{b_n -1} \left( \frac{\lfloor n y \rfloor}{n} - \frac{f_n(v)}{n} \right)$ and \newline $\frac{n}{b_n -1} \left( \frac{\lfloor n y \rfloor}{n} - u_{\beta}(R_1) \right)$ are within the interval $\left(y' - C\lambda,\, y' + C\lambda \right)$, it follows that, for any $n > N_2$ and any $v \in G_n(\lambda)$,
\begin{align}
  &\textstyle\Big| \mu_{b_n -1, q_n} \left( \frac{\tau(a_n)}{b_n -1} \le
  \frac{n}{b_n -1} \left( \frac{\lfloor n y \rfloor}{n} - \frac{f_n(v)}{n} \right) \right)\label{eq:L11e3}\\
   &\qquad\qquad\qquad \textstyle -\mu_{b_n -1, q_n} \left( \frac{\tau(a_n)}{b_n -1} \le \frac{n}{b_n -1} \left( \frac{\lfloor n y \rfloor}{n} - u_{\beta}(R_1) \right) \right) \Big| \nonumber\\
  <\, &\textstyle\mu_{b_n -1, q_n} \left( \frac{\tau(a_n)}{b_n -1} \in \left(y' - C\lambda,\, y' + C\lambda \right) \right). \nonumber
\end{align}
By (\ref{eq:L11D}) and Lemma \ref{L3}, there exists $N_3 > 0$ such that for all $n > N_3$,
\begin{equation}
\mu_{b_n -1, q_n} \left( \frac{\tau(a_n)}{b_n -1} \in \left(y' - C\lambda,\, y' + C\lambda \right) \right)
 < 4C\lambda e^{b |\beta|}. \label{eq:L11e4}
\end{equation}
Therefore, we can fix $\lambda = \frac{\epsilon}{24C} e^{- b |\beta|}$ in the first place.
Then, by (\ref{eq:L11e0}), (\ref{eq:L11e3}) and (\ref{eq:L11e4}), for any $n > \max{(N_1, N_2, N_3)}$ and any $v \in G_n(\lambda)$,
\begin{align}
  &\textstyle\Big|\mu_{b_n -1, q_n} \left( \frac{\tau(a_n)}{b_n -1} \le
  \frac{n}{b_n -1} \left( \frac{\lfloor n y \rfloor}{n} - \frac{f_n(v)}{n} \right) \right) -
    \int_{0}^{y'} u\left(\frac{a}{b}, t, b \beta\right)dt \,\Big| \label{eq:L11z} \\
   <\ & \frac{\epsilon}{6} + \frac{\epsilon}{6} = \frac{\epsilon}{3}. \nonumber
\end{align}
(\ref{eq:L11c}) follows by (\ref{eq:L11d}) and (\ref{eq:L11z}).
\end{proof}
We are in the position to show (\ref{eq:L11a}), which completes the proof of Lemma~\ref{L11}.
\begin{proof}[Proof of (\ref{eq:L11a})]
For simplicity, let $I \coloneqq \int_{0}^{y'} u\left(\frac{a}{b}, t, b \beta\right)dt$. Since
\[
y' = \frac{1}{b} u_{\beta}(R_0) \le \frac{1}{b} u_{\beta}([0 , b]\times[0, 1]) = 1,
\]
we have
\[
I = \int_{0}^{y'} u\left(\frac{a}{b}, t, b \beta\right)dt \le \int_{0}^{1} u\left(\frac{a}{b}, t, b \beta\right)dt =1.
\]
Then, given $\epsilon >0$, fix the value of $\lambda$ such that (\ref{eq:L11c}) holds for any $n > N_1$ and any $v \in G_n(\lambda)$.
By (\ref{eq:L11b}), there exists $N_2>0$ such that for any $n > N_2$,
\begin{equation}
\n \left( \pi([b_n, n]) \notin G_n(\lambda) \right) < \frac{\epsilon}{3}. \label{eq:L11e5}
\end{equation}
Then, for any $n > \max{(N_1, N_2)}$,
\begin{align*}
  &\textstyle \left| \n \left( \frac{\pi(a_n)}{n} \le y \right) - I \right|\\
 =&\, \textstyle \Big|\sum_{v \in K_n} \n \left( \frac{\pi(a_n)}{n} \le y \ \big|\  \pi([b_n, n]) = v \right)\cdot \n \left( \pi([b_n, n]) = v \right)\\
  &\qquad\qquad\qquad\qquad\qquad\qquad\qquad\quad - \textstyle\sum_{v \in K_n} I \cdot \n \left( \pi([b_n, n]) = v \right)  \, \Big|\\
 \le& \textstyle\sum_{v \in G_n(\lambda)}\Big|\n \left( \frac{\pi(a_n)}{n} \le y \ \big|\  \pi([b_n, n]) = v \right) - I\, \Big| \cdot
   \n \left( \pi([b_n, n]) = v \right) \\
  &\qquad\textstyle  + \sum_{v \notin G_n(\lambda)} \n \left( \frac{\pi(a_n)}{n} \le y \ \big|\  \pi([b_n, n]) = v \right)\cdot
   \n \left( \pi([b_n, n]) = v \right)\\
  &\qquad \textstyle + \sum_{v \notin G_n(\lambda)} I \cdot \n \left( \pi([b_n, n]) = v \right)\\
 \le&\ \textstyle \frac{\epsilon}{3}\cdot \sum_{v \in G_n(\lambda)} \n \left( \pi([b_n, n]) = v \right) +
   2 \cdot \sum_{v \notin G_n(\lambda)} \n \left( \pi([b_n, n]) = v \right)\\
 <&\ \textstyle \frac{\epsilon}{3} + \frac{2\epsilon}{3} = \epsilon.
\end{align*}
Here we use (\ref{eq:L11c}) and (\ref{eq:L11e5}) in the second to last inequality.
\end{proof}


\begin{lemma}\label{L12}
For any $0 \le a < b \le 1$ and any $\beta \in \mathbb{R}$, suppose we have sequences $\{a_n\}$, $\{b_n\}$ and $\{q_n\}$ such that
$a_n \in [n]$, $b_n \in [n]$, $q_n > 0$ and
\[
\lim_{n \to \infty} \frac{a_n}{n} = a, \qquad \lim_{n \to \infty} \frac{b_n}{n} = b, \qquad \lim_{n \to \infty} n (1 - q_n) = \beta.
\]
Then, for any $A = [y_1, y_2] \subset [0, 1]$ and $B = [y_3, y_4] \subset [0, 1]$,
\[
\textstyle\lim\limits_{n \to \infty}\n\big( \mathds{1}_A(\frac{\pi(a_n)}{n})\mathds{1}_B(\frac{\pi(b_n)}{n})\big) -
 \n \big(\mathds{1}_A(\frac{\pi(a_n)}{n})\big)\n\big(\mathds{1}_B(\frac{\pi(b_n)}{n})\big) = 0.
\]
\end{lemma}

\begin{proof}
The proof is similar to the proof of Lemma \ref{L11}, and we inherit those definitions in the previous proof.
First of all, since
\begin{align*}
  &\textstyle\n\big( \mathds{1}_A(\frac{\pi(a_n)}{n})\mathds{1}_B(\frac{\pi(b_n)}{n})\big)\\
  =\, &\textstyle\n\big( \mathds{1}_{[0, y_2]}(\frac{\pi(a_n)}{n})\mathds{1}_B(\frac{\pi(b_n)}{n})\big) -
       \n\big( \mathds{1}_{[0, y_1)}(\frac{\pi(a_n)}{n})\mathds{1}_B(\frac{\pi(b_n)}{n})\big)
\end{align*}
and
\begin{align*}
   \textstyle\n \big(\mathds{1}_A(\frac{\pi(a_n)}{n})\big) =\, \n \big(\mathds{1}_{[0, y_2]}(\frac{\pi(a_n)}{n})\big) -
       \n \big(\mathds{1}_{[0, y_1)}(\frac{\pi(a_n)}{n})\big),
\end{align*}
it suffices to show the cases when the interval $A$ is of the form $[0, y]$ or $[0, y)$ for any $y~\in~[0, 1]$.
Moreover, by Lemma \ref{M1}, we have
\[
\textstyle\lim\limits_{n \to \infty} \n \left(\frac{\pi(a_n)}{n} = y \right) = 0, \quad \forall y \in [0, 1].
\]
Hence, it suffices to show the case when $A = [0, y]$, for any $y \in [0, 1]$.

By Lemma \ref{L11}, define
\[
I \coloneqq \int_{0}^{y} u(a, t, \beta)\,dt = \int_{0}^{y'} u\left(\frac{a}{b}, t, b \beta\right)dt.
\]
By Lemma \ref{M1}, we have
\[
\textstyle\lim\limits_{n \to \infty} \n \big(\mathds{1}_A(\frac{\pi(a_n)}{n})\big) = \int_{0}^{y} u(a, t, \beta)\,dt = I.
\]
Hence it suffices to show the following,
\begin{equation}
\textstyle\lim\limits_{n \to \infty}\n\big( \mathds{1}(\frac{\pi(a_n)}{n} \le y)\mathds{1}_B(\frac{\pi(b_n)}{n})\big) -
 \n\big(\mathds{1}_B(\frac{\pi(b_n)}{n})\big)\cdot I = 0,   \label{eq:L12a}
\end{equation}
for any $y \in [0, 1]$.
Given $\epsilon > 0$, by (\ref{eq:L11c}), there exists $\lambda > 0$ and $N_1 > 0$ such that for any $n > N_1$ and any $v \in G_n(\lambda)$,
\begin{equation}
\textstyle\left|\, \n \left( \frac{\pi(a_n)}{n} \le y \ \middle|\  \pi([b_n, n]) \,{=}\, v \right) -
I \,\right| < \displaystyle\frac{\epsilon}{3} . \label{eq:L12b}
\end{equation}
By (\ref{eq:L11b}), there exists $N_2>0$ such that for any $n > N_2$,
\begin{equation}
\n \left( \pi([b_n, n]) \notin G_n(\lambda) \right) < \frac{\epsilon}{3}. \label{eq:L12c}
\end{equation}
Moreover, by conditioning on the value of $\pi([b_n, n])$, we have
\begin{align*}
   &\ \textstyle\n\left( \mathds{1}(\frac{\pi(a_n)}{n} \le y)\mathds{1}_B(\frac{\pi(b_n)}{n})\right)\\
 {=} &\textstyle\sum\limits_{v \in K_n}\textstyle \n\left( \mathds{1}(\frac{\pi(a_n)}{n} \le y)\mathds{1}_B(\frac{\pi(b_n)}{n}) \, \big|\,  \pi([b_n, n]) \,{=}\, v \right) \cdot
     \n \left( \pi([b_n, n]) \,{=}\, v \right)\\
 {=} &\textstyle\sum\limits_{v \in K_n} \textstyle\n\left( \mathds{1}(\frac{\pi(a_n)}{n} \le y) \,\big|\, \pi([b_n, n]) \,{=}\, v \right)\cdot
      \mathds{1}_B(\frac{v_1}{n}) \cdot \n \left( \pi([b_n, n]) \,{=}\, v \right)\\
 {=} &\textstyle\sum\limits_{v \in G_n(\lambda)} \textstyle\n\left( \mathds{1}(\frac{\pi(a_n)}{n} \le y) \, \big|\, \pi([b_n, n]) \,{=}\, v \right)\cdot
      \mathds{1}_B(\frac{v_1}{n}) \cdot \n \left( \pi([b_n, n]) \,{=}\, v \right)\\
  + &\textstyle\sum\limits_{v \notin G_n(\lambda)} \textstyle\n\left( \mathds{1}(\frac{\pi(a_n)}{n} \le y) \,\big|\, \pi([b_n, n]) \,{=}\, v \right)\cdot
      \mathds{1}_B(\frac{v_1}{n}) \cdot \n \left( \pi([b_n, n]) \,{=}\, v \right)\\
  \text{and}&\\
   &\ \textstyle\n\left(\mathds{1}_B(\frac{\pi(b_n)}{n})\right)\\
 {=}& \textstyle\sum\limits_{v \in K_n} \textstyle\n\left(\mathds{1}_B(\frac{\pi(b_n)}{n})\,\big|\, \pi([b_n, n]) \,{=}\, v \right)\cdot \n \left( \pi([b_n, n]) \,{=}\, v \right)\\
 {=}& \textstyle\sum\limits_{v \in K_n} \textstyle\mathds{1}_B(\frac{v_1}{n}) \cdot \n \left( \pi([b_n, n]) \,{=}\, v \right)\\
 {=}&\textstyle\sum\limits_{v \in G_n(\lambda)} \textstyle\mathds{1}_B(\frac{v_1}{n}) \cdot \n \left( \pi([b_n, n]) \,{=}\, v \right)\\
  &\textstyle\qquad\qquad\qquad\qquad\qquad\qquad + \sum\limits_{v \notin G_n(\lambda)} \textstyle\mathds{1}_B(\frac{v_1}{n}) \cdot \n \left( \pi([b_n, n]) \,{=}\, v \right).
\end{align*}
Here $v_1$ denotes the first entry of vector $v$. Hence, for any $n > \max{(N_1, N_2)}$, we have
\begin{align*}
  &\,\textstyle\left|\,\n\left( \mathds{1}(\frac{\pi(a_n)}{n} \le y)\mathds{1}_B(\frac{\pi(b_n)}{n})\right)
    - \n\left(\mathds{1}_B(\frac{\pi(b_n)}{n})\right)\cdot I \, \right|\\
  \le& \textstyle\sum\limits_{v \in G_n(\lambda)}\textstyle\left|\, \n\left( \mathds{1}(\frac{\pi(a_n)}{n} \le y) \, \big|\,  \pi([b_n, n]) \,{=}\, v \right) - I \, \right|
     \cdot \n \left( \pi([b_n, n]) \,{=}\, v \right)\\
     &\qquad\qquad\qquad\qquad\qquad\qquad\qquad\qquad\quad\textstyle +  2\sum\limits_{v \notin G_n(\lambda)} \n \left( \pi([b_n, n])\,{=}\, v \right)\\
  \le&\  \textstyle\frac{\epsilon}{3} \sum\limits_{v \in G_n(\lambda)} \n \left( \pi([b_n, n]) \,{=}\, v \right)
     + 2\sum\limits_{v \notin G_n(\lambda)} \n \left( \pi([b_n, n]) \,{=}\, v \right)\\
  <&\  \frac{\epsilon}{3} + 2\cdot \frac{\epsilon}{3} = \epsilon.
\end{align*}
The first inequality follows from triangle inequality and the fact that,
\[
\textstyle\n\left( \mathds{1}(\frac{\pi(a_n)}{n} \le y) \,\Big|\, \pi([b_n, n]) \,{=}\, v \right) \le 1, \quad \mathds{1}_B(\frac{v_1}{n}) \le 1,
\text{ and } I \le 1.
\]
The last two inequalities follow from (\ref{eq:L12b}) and (\ref{eq:L12c}) respectively.

\end{proof}


Before we start to prove Lemma \ref{L18} and Lemma \ref{L13}, we briefly introduce the following facts:
\begin{lemma}\label{L14}
For any $s, t, i \in [n]$,
\[
\min( q^d, q^{-d}) \le \frac{\m(\pi(s)=i)}{\m(\pi(t)=i)}\le \max( q^d, q^{-d}),
\]
where $d = |s-t|$.
\end{lemma}
\begin{lemma}\label{L15}
For any $s, t, w, i, j \in [n]$ such that either $w < \min{(s, t)}$ or $w > \max{(s, t)}$,
\[
\min( q^d, q^{-d}) \le \frac{ \m \left( \{ \pi \in S_n : \pi(s)= i \text{ and } \pi(w) = j \}\right) }
    {\m \left( \{ \pi \in S_n : \pi(t)= i \text{ and } \pi(w) = j \}\right) } \le \max( q^d, q^{-d}),
\]
where $d = |s-t|$.
\end{lemma}

These two lemmas follow from similar argument as in the proof of Lemma \ref{L2}. We omit their proofs. From these two lemmas, we can show the following,
\begin{lemma}\label{L16}
For any $A \subset [0, 1]$ and any $s, t \in [n]$,
\[
\textstyle\left|\m \left( \mathds{1}_A \big(\frac{\pi(s)}{n} \big) \right) -
     \m \left( \mathds{1}_A \big(\frac{\pi(t)}{n} \big)\right)\right| \le M,
\]
where $M = \max{(|1-q^d|, |1 - q^{-d}|)}$ and $d = |s-t|$.
\end{lemma}

\begin{lemma}\label{L17}
For any $A, B \subset [0, 1]$ and any $s, t, w \in [n]$ such that either $w < \min{(s, t)}$ or $w > \max{(s, t)}$,
\[
\textstyle\left|\m \left( \mathds{1}_A \big(\frac{\pi(s)}{n} \big)\mathds{1}_B \big(\frac{\pi(w)}{n} \big) \right) -
     \m \left( \mathds{1}_A \big(\frac{\pi(t)}{n} \big) \mathds{1}_B \big(\frac{\pi(w)}{n} \big) \right) \right| \le M,
\]
where $M = \max{(|1-q^d|, |1 - q^{-d}|)}$ and $d = |s-t|$.
\end{lemma}
Here we only deduce Lemma \ref{L16} from Lemma \ref{L14}. Lemma \ref{L17} follows from Lemma \ref{L15} by the similar argument.
\begin{proof}[Proof of Lemma \ref{L16}]
Without loss of generality, assume $0 < q < 1$. By Lemma \ref{L14}, for any $i \in [n]$, we have
\[
q^d \le \frac{\m(\pi(s)=i)}{\m(\pi(t)=i)}\le q^{-d}.
\]
Hence
\[
\textstyle q^d \sum\limits_{\{i:\frac{i}{n} \in A\}} \m(\pi(t)=i) \le \sum\limits_{\{i:\frac{i}{n} \in A\}} \m(\pi(s)=i)\le q^{-d}\sum\limits_{\{i:\frac{i}{n} \in A\}} \m(\pi(t)=i).
\]
Thus
\[
\textstyle\m \left( \mathds{1}_A \big(\frac{\pi(t)}{n} \big) \right) \cdot q^d \le \m \left( \mathds{1}_A \big(\frac{\pi(s)}{n} \big) \right)
  \le \m \left( \mathds{1}_A \big(\frac{\pi(t)}{n} \big) \right) \cdot q^{-d}.
\]
Therefore
\begin{align*}
   &\textstyle\left|\m \left( \mathds{1}_A \big(\frac{\pi(s)}{n} \big) \right) -
     \m \left( \mathds{1}_A \big(\frac{\pi(t)}{n} \big) \right) \right|\\
 \le&\, \textstyle\m \left( \mathds{1}_A \big(\frac{\pi(t)}{n} \big) \right)\max{(q^{-d}-1, 1 - q^d)}\\
 \le&\, \textstyle\max{(q^{-d}-1, 1 - q^d)}.
\end{align*}
\end{proof}


\begin{proof}[Proof of Lemma \ref{L13}]
Let $m$ be a positive integer whose value is to be determined. Define the following $m+1$ sequences $\{ a^{(k)}_n \}$, $0 \le k \le m$, as follows,
\begin{equation}
a^{(k)}_n \coloneqq 
\begin{cases}
1, &\text{if $k = 0$;}\\
\left\lceil\frac{k n}{m}\right\rceil, &\text{if $1 \le k \le m$}.
\end{cases} 
\label{eq:L13a}
\end{equation}
Then, for any $0 \le k \le m$, we have $\lim_{n \to \infty}\frac{ a^{(k)}_n}{n} = \frac{k}{m}$.
Also, for any $0 \le k \le m-1$ and $n > m$ we have
\[
1 \le a^{(k+1)}_n - a^{(k)}_n \le \frac{n}{m} + 1.
\]
Then, for any $n > m$ and any $i, j \in [n]$ with $ i < j$, there exist unique $k$ and $l$ such that
\begin{equation}
 i \in \left[a^{(k)}_n, a^{(k+1)}_n \right), \quad \text{ and }\quad j \in \left(a^{(l-1)}_n, a^{(l)}_n \right]. \label{eq:L13b}
\end{equation}
Clearly, we have
\begin{equation}
 k  < l, \qquad  |i - a^{(k)}_n| \le \frac{n}{m} \quad \text{ and } \quad |j - a^{(l)}_n| \le \frac{n}{m}. \label{eq:L13c}
\end{equation}

Then, given $\epsilon > 0$, fix a sufficiently large $m$ in the first place such that,
\[
\left|e^{\frac{\beta}{m}} - 1 \right| < \frac{\epsilon}{12}, \qquad
\left|e^{-\frac{\beta}{m}} - 1 \right| < \frac{\epsilon}{12}.
\]
Next, since $\lim_{n \to \infty} q_n^n = e^{-\beta}$, there exists $N_1> 0$ such that for any $n > N_1$,
\[
\Big|e^{\frac{\beta}{m}} - q_{n}^{- \frac{n}{m}} \Big| < \frac{\epsilon}{12}, \qquad
\Big|e^{-\frac{\beta}{m}} - q_{n}^{\frac{n}{m}} \Big| < \frac{\epsilon}{12}.
\]
Then, by triangle inequality, for any $n > N_1$,
\begin{equation}
\max{\left(\left|1-q_n^{\frac{n}{m}}\right|, \left|1-q_n^{-\frac{n}{m}}\right|\right)} < \frac{\epsilon}{6}. \label{eq:L13cc}
\end{equation}
For the simplicity of notation, define
\begin{align*}
U &\coloneqq \textstyle\Big|\n \big(\mathds{1}_A\big(\frac{\pi(i)}{n}\big)\big) \n \big(\mathds{1}_B\big(\frac{\pi(j)}{n}\big)\big) -
     \n \big(\mathds{1}_A\big(\frac{\pi(a_n^{(k)})}{n}\big)\big) \n \big(\mathds{1}_B\big(\frac{\pi(j)}{n}\big)\big) \Big|,\\
V &\coloneqq \textstyle\Big|\n \big(\mathds{1}_A\big(\frac{\pi(i)}{n}\big)\mathds{1}_B\big(\frac{\pi(j)}{n}\big)\big) -
   \n \big(\mathds{1}_A\big(\frac{\pi(a_n^{(k)})}{n}\big)\mathds{1}_B\big(\frac{\pi(j)}{n}\big)\big) \Big|,\\
W &\coloneqq \textstyle \Big| \n \big(\mathds{1}_A\big(\frac{\pi(a_n^{(k)})}{n}\big)\mathds{1}_B\big(\frac{\pi(j)}{n}\big)\big) -
     \n \big(\mathds{1}_A\big(\frac{\pi(a_n^{(k)})}{n}\big)\big) \n \big(\mathds{1}_B\big(\frac{\pi(j)}{n}\big)\big) \Big|\\
  &\,= \textstyle\Big|\cov_n\Big(\mathds{1}_A\big(\frac{\pi(a_n^{(k)})}{n}\big), \mathds{1}_B\big(\frac{\pi(j)}{n}\big)\Big)\Big|.
\end{align*}
Then, by (\ref{eq:L13c}), (\ref{eq:L13cc}), Lemma \ref{L16} and Lemma \ref{L17}, for any $n > \max{(m, N_1)}$ and any $0 \le i < j \le n$ with
corresponding $k, l$ defined in (\ref{eq:L13b}), we have
\begin{align*}
  U =&\, \textstyle\n \big(\mathds{1}_B\big(\frac{\pi(j)}{n}\big)\big) \cdot \Big| \n \big(\mathds{1}_A\big(\frac{\pi(i)}{n}\big)\big) -
   \n \big(\mathds{1}_A\big(\frac{\pi(a_n^{(k)})}{n}\big)\big)\Big|\\
 \le&\,\textstyle\Big| \n \big(\mathds{1}_A\big(\frac{\pi(i)}{n}\big)\big) - \n \big(\mathds{1}_A\big(\frac{\pi(a_n^{(k)})}{n}\big)\big)\Big|\\
 \le&\, \textstyle\max{\left(\left|1-q_n^{\frac{n}{m}}\right|, \left|1-q_n^{-\frac{n}{m}}\right|\right)} < \textstyle\frac{\epsilon}{6},\\
  \textstyle V \le&\, \textstyle\max{\left(\left|1-q_n^{\frac{n}{m}}\right|, \left|1-q_n^{-\frac{n}{m}}\right|\right)} < \frac{\epsilon}{6}.
\end{align*}
Whence, again, by triangle inequality, for any $n > \max{(m, N_1)}$,
\begin{align}
  &\textstyle\Big|\cov_n\Big(\mathds{1}_A\big(\frac{\pi(i)}{n}\big), \mathds{1}_B\big(\frac{\pi(j)}{n}\big)\Big)\Big| \label{eq:L13d}\\
  =\,&\textstyle\Big| \n \big(\mathds{1}_A\big(\frac{\pi(i)}{n}\big)\mathds{1}_B\big(\frac{\pi(j)}{n}\big)\big) -
     \n \big(\mathds{1}_A\big(\frac{\pi(i)}{n}\big)\big) \n \big(\mathds{1}_B\big(\frac{\pi(j)}{n}\big)\big)\Big| \nonumber\\
  <&\ \textstyle U + V + W \nonumber\\
  <&\ \textstyle \Big|\cov_n\Big(\mathds{1}_A\big(\frac{\pi(a_n^{(k)})}{n}\big), \mathds{1}_B\big(\frac{\pi(j)}{n}\big)\Big)\Big| + \frac{\epsilon}{3}.  \nonumber
\end{align}
By the same argument, it follows that for any $n > \max{(m, N_1)}$,
\begin{align}
  &\textstyle\Big|\cov_n\Big(\mathds{1}_A\big(\frac{\pi(a_n^{(k)})}{n}\big), \mathds{1}_B\big(\frac{\pi(j)}{n}\big)\Big)\Big| \label{eq:L13e}\\
  <&\textstyle\Big| \cov_n \Big(\mathds{1}_A\big(\frac{\pi(a_n^{(k)})}{n}\big), \mathds{1}_B\big(\frac{\pi(a_n^{(l)})}{n}\big)\Big) \Big| +\frac{\epsilon}{3}. \nonumber
\end{align}
Combining (\ref{eq:L13d}) and (\ref{eq:L13e}), for any $n > \max{(m, N_1)}$ and any $0 \le i < j \le n$ with
corresponding $k, l$ defined in (\ref{eq:L13b}), we have
\begin{align*}
  &\textstyle\Big|\cov_n\Big(\mathds{1}_A\big(\frac{\pi(i)}{n}\big), \mathds{1}_B\big(\frac{\pi(j)}{n}\big)\Big)\Big|\\
  <\,&\textstyle\Big| \cov_n \Big(\mathds{1}_A\big(\frac{\pi(a_n^{(k)})}{n}\big), \mathds{1}_B\big(\frac{\pi(a_n^{(l)})}{n}\big)\Big) \Big| +\frac{2\epsilon}{3}.
\end{align*}
Moreover, since $m$ is fixed, by Lemma \ref{L12}, there exists $N_2 >0$ such that, for any $n > N_2$ and any $0 \le k < l \le m$, we have
\[
\textstyle\Big| \cov_n \Big(\mathds{1}_A\big(\frac{\pi(a_n^{(k)})}{n}\big), \mathds{1}_B\big(\frac{\pi(a_n^{(l)})}{n}\big)\Big) \Big| < \frac{\epsilon}{3}.
\]
Thus, for $n > \max{(m, N_1, N_2)}$ and any $0 \le i < j \le n$,
\[
\textstyle\Big| \n \big(\mathds{1}_A\big(\frac{\pi(i)}{n}\big)\mathds{1}_B\big(\frac{\pi(j)}{n}\big)\big) -
     \n \big(\mathds{1}_A\big(\frac{\pi(i)}{n}\big)\big) \n \big(\mathds{1}_B\big(\frac{\pi(j)}{n}\big)\big)\Big|
     < \epsilon.
\]
\end{proof}

\begin{proof}[Proof of Lemma \ref{L18}]
The proof of Lemma \ref{L18} is similar to the proof of Lemma \ref{L13}. 
Firstly, since $u(x, y, \beta)$ is uniformly continuous on $[0, 1] \times [0, 1]$, given $\epsilon > 0$, there exists $m_1 > 0$ such that
\[
\sup_{ \substack{|s - t|< \frac{1}{m_1}\\
                 s, t, y\in [0, 1]}}
                 |u(s, y, \beta) - u(t, y, \beta)| < \frac{\epsilon}{6}.
\]
Hence, for any $|s - t|< \frac{1}{m_1}$ with $s, t\in [0, 1]$, we have
\begin{align}
    &\left|\int_{y_1}^{y_2} u(s, y, \beta)\, dy - \int_{y_1}^{y_2} u(t, y, \beta)\, dy \right|  \label{eq:L18a}\\
 \le&\int_{y_1}^{y_2}|u(s, y, \beta) - u(t, y, \beta)|\, dy  \nonumber\\
 \le&\int_{0}^{1}|u(s, y, \beta) - u(t, y, \beta)|\, dy  \nonumber\\
  <&\, \frac{\epsilon}{6}. \nonumber
\end{align}
Then, choose an $m > 2 m_1$ such that
\[
\left|e^{\frac{\beta}{m}} - 1 \right| < \frac{\epsilon}{12}, \qquad
\left|e^{-\frac{\beta}{m}} - 1 \right| < \frac{\epsilon}{12}.
\]
Next, since $\lim_{n \to \infty} q_n^n = e^{-\beta}$, there exists $N_1> 0$ such that for any $n > N_1$,
\[
\Big|e^{\frac{\beta}{m}} - q_{n}^{- \frac{n}{m}} \Big| < \frac{\epsilon}{12}, \qquad
\Big|e^{-\frac{\beta}{m}} - q_{n}^{\frac{n}{m}} \Big| < \frac{\epsilon}{12}.
\]
By triangle inequality, for any $n > N_1$,
\[
\max{\left(\left|1-q_n^{\frac{n}{m}}\right|, \left|1-q_n^{-\frac{n}{m}}\right|\right)} < \frac{\epsilon}{6}.
\]
Next, define the $m+1$ sequences $\{ a^{(k)}_n \}$, $0 \le k \le m$, as in (\ref{eq:L13a}). By (\ref{eq:L13c}) and Lemma \ref{L16}, for any $n > \max{(m, N_1)}$ and any $i \in [n]$ with
corresponding $k$ defined in (\ref{eq:L13b}), we have
\begin{align}
   &\textstyle\Big| \n \big(\mathds{1}_A\big(\frac{\pi(i)}{n}\big)\big) - \n \big(\mathds{1}_A\big(\frac{\pi(a_n^{(k)})}{n}\big)\big)\Big| \label{eq:L18b}\\
 \le&\, \textstyle\max{\left(\left|1-q_n^{\frac{n}{m}}\right|, \left|1-q_n^{-\frac{n}{m}}\right|\right)}  \nonumber\\
 <&\, \frac{\epsilon}{6}. \nonumber
\end{align}
Secondly, by the definition of $a^{(k)}_n$ in (\ref{eq:L13a}), it is easily seen that
\[
\frac{k n}{m} \le a^{(k)}_n \le \frac{k n}{m} +1.
\]
Thus, for any $n > m$ and any $i \in [n]$ with corresponding $k$ defined in (\ref{eq:L13b}), we have
\begin{align*}
  &\frac{k n}{m} \le a^{(k)}_n \le i < a^{(k+1)}_n \le \frac{(k+1) n}{m} + 1\\
  \Rightarrow\ & \frac{k}{m} \le \frac{i}{n} \le \frac{k+1}{m} + \frac{1}{n}\\
  \Rightarrow\ & \left| \frac{i}{n} - \frac{k}{m} \right| \le \frac{1}{m} + \frac{1}{n} < \frac{2}{m} < \frac{1}{m_1}.
\end{align*}
Hence, by (\ref{eq:L18a}), for any $n > m$ and any $i \in [n]$ with corresponding $k$ defined in (\ref{eq:L13b}), we have
\begin{equation}
\Big|\int_{y_1}^{y_2} u\Big(\frac{i}{n}, y, \beta\Big)\, dy - \int_{y_1}^{y_2} u\Big(\frac{k}{m}, y, \beta\Big)\, dy \,\Big|
 < \frac{\epsilon}{6}. \label{eq:L18c}
\end{equation}
Thirdly, since $\lim_{n \to \infty} \frac{a_n^{(k)}}{n} = \frac{k}{m}$ for any $0 \le k \le m$, by Lemma \ref{M1}, there exists $N_1 > 0$
such that, for any $n > N_1$ and any $0 \le k \le m$,
\begin{equation}
\bigg|\n \bigg(\mathds{1}_A\bigg(\frac{\pi(a_n^{(k)})}{n}\bigg)\bigg) - \int_{y_1}^{y_2} u\Big(\frac{k}{m}, y, \beta\Big)\, dy \,\bigg| < \frac{\epsilon}{3}. \label{eq:L18e}
\end{equation}
Therefore, for any $n > \max{(m, N_1, N_2)}$ and any $i \in [n]$ with corresponding $k$ defined in (\ref{eq:L13b}), we have
\begin{align}
   &\textstyle\Big|\n \big(\mathds{1}_A\big(\frac{\pi(i)}{n}\big)\big) - \int_{y_1}^{y_2} u\big(\frac{i}{n}, y, \beta\big)\, dy\, \Big| \nonumber\\
 \le&\, \textstyle\Big|\n \big(\mathds{1}_A\big(\frac{\pi(i)}{n}\big)\big) - \n \big(\mathds{1}_A\big(\frac{\pi(a_n^{(k)})}{n}\big)\big) \Big| \nonumber\\
  &\qquad\qquad\qquad\qquad \textstyle + \Big|\n \big(\mathds{1}_A\big(\frac{\pi(a_n^{(k)})}{n}\big)\big) - \int_{y_1}^{y_2} u\big(\frac{k}{m}, y, \beta\big)\, dy \,\Big| \nonumber\\
  &\qquad\qquad\qquad\qquad \textstyle + \Big|\int_{y_1}^{y_2} u\big(\frac{k}{m}, y, \beta\big)\, dy - \int_{y_1}^{y_2} u\big(\frac{i}{n}, y, \beta\big)\, dy \,\Big|  \nonumber\\
  <&\, \frac{\epsilon}{3} + \frac{\epsilon}{6} + \frac{\epsilon}{6} < \epsilon. \nonumber
\end{align}
The last inequality follows from (\ref{eq:L18b}), (\ref{eq:L18c}) and (\ref{eq:L18e}).
\end{proof}

\section{Proof of Theorem \ref{M0}}
In this section, we show Theorem \ref{M0} using Lemma \ref{L18} and Lemma \ref{L13}. In the proof we approximate the continuous function $f$ on $[0, 1]$ by a sequence of simple functions. The following elementary lemma will be used in the proof.

\begin{lemma}\label{LM0}
Given random variables $X, X', Y, Y'$ such that $|X - X'| < \epsilon$, $|Y - Y'| < \epsilon$ and $\max(|X|, |X'|, |Y|, |Y'|) < M$, we have
\[
|\cov(X, Y) - \cov(X', Y')| < 4M \epsilon.
\]
\end{lemma}

\begin{proof}
Since $XY - X'Y' = X(Y - Y') + Y'(X - X')$, we have
\begin{equation}\label{eq:LM01}
|XY - X'Y'| \le |X||Y-Y'| + |Y'||X - X'| < 2M \epsilon.
\end{equation}
Similarly, since $\E(X) \E(Y) - \E(X')\E(Y') = \E(X)\E(Y - Y') + \E(Y')\E(X - X')$, we have
\begin{equation}\label{eq:LM02}
|\E(X) \E(Y) - \E(X')\E(Y') | < 2M \epsilon.
\end{equation}
Hence
\begin{align*}
|\cov(X, Y) - \cov(X', Y')| &\le \E|XY - X'Y'| + |\E(X) \E(Y) - \E(X')\E(Y')| \\
                            &< 4M \epsilon.
\end{align*}
\end{proof}

\begin{proof}[Proof of Theorem \ref{M0}]
Given continuous function $f : [0, 1] \longrightarrow \mathbb{R}$, define a sequence of simple functions $\{g_m\}_{m \ge 1}$ as follows,
\[
\textstyle g_m(x) \coloneqq \sum_{k = 1}^{m} f\left(\frac{k}{m}\right)\mathds{1}_{A_k}(x),  \quad\text{ where } A_k \coloneqq \left(\frac{k-1}{m}, \frac{k}{m}\right],
\]
and $g_m(0) \coloneqq f(0)$. Since $f$ is continuous on a compact interval, it is uniformly continuous on $[0, 1]$ and there exists $M > 0$ such that $|f(x)| < M$. Hence, for any $\epsilon > 0$, there exists an $N > 0$ such that for any $m > N$ we have
\[
|f(x) - g_m(x)| < \epsilon, \quad \forall x \in [0, 1].
\]
Hence for any $m > N$ and any $i \in [n]$ we have
\begin{equation}\label{eq:M03}
\textstyle\left|\n \left( f\left(\frac{\pi(i)}{n}\right) \right) - \n \left( g_m\left(\frac{\pi(i)}{n}\right) \right)\right| < \epsilon,
\end{equation}
and
\begin{equation}\label{eq:M04}
\textstyle\left|\int_0^1 f(y)\cdot u\left(\frac{i}{n}, y, \beta \right)\, dy - \int_0^1 g_m(y)\cdot u\left(\frac{i}{n}, y, \beta \right)\, dy \right| < \epsilon.
\end{equation}
Moreover, we have
\begin{align*}
 &\textstyle\left|\n \left( g_m\left(\frac{\pi(i)}{n}\right) \right) - \int_0^1 g_m(y)\cdot u\left(\frac{i}{n}, y, \beta \right)\, dy \right|\\
\le\ &\sum_{k = 1}^{m}\textstyle|f(\frac{k}{m})|\left|\n \left( \mathds{1}_{A_k}\left(\frac{\pi(i)}{n}\right) \right) - \int_0^1 \mathds{1}_{A_k}(y)\cdot u\left(\frac{i}{n}, y, \beta \right)\, dy \right|\\
\le\ & \sum_{k = 1}^{m}\textstyle M\left|\n \left( \mathds{1}_{A_k}\left(\frac{\pi(i)}{n}\right) \right) - \int_{A_k} u\left(\frac{i}{n}, y, \beta \right)\, dy \right|,
\end{align*}
Hence by triangle inequality, the first claim (\ref{eq:M01}) follows from Lemma \ref{L18}, (\ref{eq:M03}) and (\ref{eq:M04}). To prove the second claim (\ref{eq:M02}), we use the same technique by approximating $f$ by simple functions $g_m$. Note that by Lemma \ref{LM0}, for any $m > N$ and any $1 \le i < j \le n$, we have
\begin{equation}\label{eq:M05}
\textstyle\left| \cov_n\Big(f\big(\frac{\pi(i)}{n}\big), f\big(\frac{\pi(j)}{n}\big)\Big)  -  
\cov_n\Big(g_m\big(\frac{\pi(i)}{n}\big), g_m\big(\frac{\pi(j)}{n}\big)\Big)\right| < 4M \epsilon.
\end{equation}
Note that
\begin{align*}
& \textstyle \left|\cov_n\Big(g_m\big(\frac{\pi(i)}{n}\big), g_m\big(\frac{\pi(j)}{n}\big)\Big)\right| \\
=\ & \left|\sum_{k = 1}^m \sum_{l = 1}^m \textstyle f\left(\frac{k}{m}\right)f\left(\frac{l}{m}\right) \cov_n\Big(\mathds{1}_{A_k}\big(\frac{\pi(i)}{n}\big), \mathds{1}_{A_l}\big(\frac{\pi(j)}{n}\big)\Big)\right|\\
\le\ & M^2 \sum_{k = 1}^m \sum_{l = 1}^m \textstyle \left| \cov_n\Big(\mathds{1}_{A_k}\big(\frac{\pi(i)}{n}\big), \mathds{1}_{A_l}\big(\frac{\pi(j)}{n}\big)\Big)\right|.
\end{align*}
(\ref{eq:M02}) follows from Lemma \ref{L13} and (\ref{eq:M05}).
\end{proof}


\section{The Convergence of the Empirical Measure}
Recall that, under the conditions in Theorem \ref{M2}, we need to show the convergence of the empirical measure induced by $\{(\frac{i}{n}, \frac{\tau\ccirc\pi(i)}{n})\}_{i \in [n]}$.
Note that, by relabeling the indices, we have $\{(\frac{i}{n}, \frac{\tau\ccirc\pi(i)}{n})\}_{i \in [n]} = \{(\frac{\pi^{-1}(i)}{n}, \frac{\tau(i)}{n})\}_{i \in [n]}$.  Since $\pi$ and $\tau$ are independent, for a given $i$, the $x$ coordinate and $y$ coordinate of $\big(\frac{\pi^{-1}(i)}{n}, \frac{\tau(i)}{n}\big)$ are independent. We will exploit this property to establish the first and second moment estimates of the number of these points which fall inside a given rectangle. Recall that, in Section \ref{S2}, for any $\pi \in S_n$, we define $L_{\pi}$ as the empirical probability measure of $\{(\frac{i}{n}, \frac{\pi(i)}{n})\}_{i \in [n]}$, i.e.,
\[
L_{\pi}(R) \coloneqq \frac{1}{n}\sum_{i = 1}^{n} \mathds{1}_R\Big( \frac{i}{n}, \frac{\pi(i)}{n} \Big), \qquad \forall R \in  \mathcal{B}_{[0, 1]\times[0,1]}.
\]
Similarly, we now define $L_{\pi, \tau}$ to be the empirical probability measure of $\{(\frac{\pi(i)}{n}, \frac{\tau(i)}{n})\}_{i \in [n]}$. That is
\[
L_{\pi, \tau}(R) \coloneqq \frac{1}{n}\sum_{i = 1}^{n} \mathds{1}_R\Big(\frac{\pi(i)}{n}, \frac{\tau(i)}{n}\Big), \qquad \forall R \in  \mathcal{B}_{[0, 1]\times[0,1]}.
\]


Lemma \ref{L18} and Lemma \ref{L13} imply the following weak convergence for $L_{\pi,\tau}$.

\begin{lemma}\label{L100}
Under the same conditions as Theorem \ref{M2}, for any $R = (x_1, x_2]\times (y_1, y_2] \subset [0, 1]\times [0, 1]$, we have
\begin{equation}
\lim_{n \to \infty} \p_n \left(\left|\,L_{\pi, \tau}(R) - \int_{R} \rho(x, y)\, dxdy\, \right| > \epsilon \right) = 0   \label{eq:L100a}
\end{equation}
for any $\epsilon > 0$. Here $\rho(x, y)$ is the density function defined in Theorem \ref{M2}.
\end{lemma}

\begin{proof}
Let $\bar{R} = [x_1, x_2]\times[y_1, y_2]$ be the closure of $R$. Since, for any vertical or horizontal line $l$ and any $\pi, \tau \in S_n$, we have $L_{\pi, \tau}(l) \le \frac{1}{n}$, it follows that
\[
\left|\,L_{\pi, \tau}(R) - L_{\pi, \tau}(\bar{R}) \, \right| \le \frac{2}{n}.
\]
Then, given $\epsilon > 0$, for any $n > \frac{4}{\epsilon}$, by triangle inequality and the fact that $\int_{R} \rho(x, y)\, dxdy = \int_{\bar{R}} \rho(x, y)\, dxdy$, we get
\begin{align*}
&\textstyle \left|\,L_{\pi, \tau}(R) - \int_{R} \rho(x, y)\, dxdy\, \right| > \epsilon\\
\Rightarrow \ &\textstyle\left|\,L_{\pi, \tau}(\bar{R}) - \int_{\bar{R}} \rho(x, y)\, dxdy\, \right| > \frac{\epsilon}{2}.
\end{align*}
Hence, it suffices to show (\ref{eq:L100a}) for $R = [x_1, x_2]\times[y_1, y_2]$. In the remainder of the proof, let $R \coloneqq [x_1, x_2]\times[y_1, y_2]$. We will show
\begin{align}
 &\lim_{n \to \infty}\E_n (L_{\pi, \tau}(R)) = \int_{R} \rho(x, y)\, dxdy,  \label{eq:L100b} \\
 &\lim_{n \to \infty}\var_n (L_{\pi, \tau}(R)) = 0.  \label{eq:L100c}
\end{align}
Then, (\ref{eq:L100a}) follows from (\ref{eq:L100b}) and (\ref{eq:L100c}) by Chebyshev's inequality and triangle inequality.

Let $A = [x_1, x_2]$ and $B = [y_1, y_2]$. Define
\begin{align*}
\delta_n^{(i)} &\coloneqq \textstyle\n \big(\mathds{1}_A\big(\frac{\pi(i)}{n}\big) \big) - \int_A u\big(x, \frac{i}{n}, \beta \big)\, dx, \\
\delta_n^{'(i)} &\coloneqq \textstyle\mu_{n, q'_n} \big(\mathds{1}_B\big(\frac{\tau(i)}{n}\big) \big) -
       \int_B u\big(\frac{i}{n}, y, \gamma \big)\, dy,
\end{align*}
\[
\delta_n \coloneqq \max_{i \in [n]}( |\delta_n^{(i)}|) \quad \text{and} \quad
\delta'_n \coloneqq \max_{i \in [n]}( |\delta_n^{'(i)}|).
\]
Then, by Lemma \ref{L18} and the fact that $u(x, y, \beta) = u(y, x, \beta)$, for any $\epsilon > 0$, there exists $N_1 > 0$ such that,
for any $n > N_1$,
\[
\delta_n < \frac{\epsilon}{3} \qquad \text{and} \qquad \delta'_n < \frac{\epsilon}{3}.
\]
Without loss of generality, assume $0 < \epsilon < 1$. Then, for any $n > N_1$ and any $i \in [n]$, we have
\begin{align}
   &\textstyle\Big| \n \big(\mathds{1}_A\big(\frac{\pi(i)}{n}\big) \big) \mu_{n, q'_n} \big(\mathds{1}_B\big(\frac{\tau(i)}{n}\big) \big) -
     \int_R u\big(x, \frac{i}{n}, \beta \big)u\big(\frac{i}{n}, y, \gamma \big)\, dxdy \Big| \label{eq:L100d}\\
  =\,&\textstyle\Big|\Big(\delta_n^{(i)} + \int_A u\big(x, \frac{i}{n}, \beta \big)\, dx \Big)
           \Big(\delta_n^{'(i)} + \int_B u\big(\frac{i}{n}, y, \gamma \big)\, dy \Big) \nonumber\\
    &\textstyle\qquad\qquad\qquad\qquad\qquad
    - \int_A u\big(x, \frac{i}{n}, \beta \big)\, dx \cdot \int_B u\big(\frac{i}{n}, y, \gamma \big)\, dy \Big| \nonumber\\
  \le\,&\textstyle \Big|\delta_n^{'(i)}\Big| \int_A u\big(x, \frac{i}{n}, \beta \big)\, dx +
         \Big|\delta_n^{(i)}\Big| \int_B u\big(\frac{i}{n}, y, \gamma \big)\, dy + \Big|\delta_n^{(i)}\delta_n^{'(i)}\Big| \nonumber\\
  <\,&\textstyle \frac{\epsilon}{3} + \frac{\epsilon}{3} + \frac{\epsilon}{3} = \epsilon. \nonumber
\end{align}
Here we use Lemma \ref{L9} in the last inequality. Hence, for any $n > N_1$,
\begin{align}
  &\textstyle\bigg|\, \E_n (L_{\pi, \tau}(R)) -
         \frac{1}{n}\sum_{i=1}^{n}\int_R u\big(x, \frac{i}{n}, \beta \big)u\big(\frac{i}{n}, y, \gamma \big)\, dxdy \, \bigg|  \label{eq:L100e}\\
  =\,&\textstyle\bigg|\, \frac{1}{n}\sum_{i=1}^{n} \E_n \Big(\mathds{1}_R\big(\frac{\pi(i)}{n}, \frac{\tau(i)}{n}\big)\Big) -
           \frac{1}{n}\sum_{i=1}^{n}\int_R u\big(x, \frac{i}{n}, \beta \big)u\big(\frac{i}{n}, y, \gamma \big)\, dxdy \, \bigg| \nonumber\\
  \le\,&\textstyle \frac{1}{n}\sum_{i=1}^{n} \Big| \, \E_n \Big(\mathds{1}_A\big(\frac{\pi(i)}{n}\big) \mathds{1}_B\big(\frac{\tau(i)}{n}\big)\Big) -
         \int_R u\big(x, \frac{i}{n}, \beta \big)u\big(\frac{i}{n}, y, \gamma \big)\, dxdy \, \Big| \nonumber\\
  =\,&\textstyle \frac{1}{n}\sum_{i=1}^{n} \Big| \n \big(\mathds{1}_A\big(\frac{\pi(i)}{n}\big) \big) \mu_{n, q'_n} \big(\mathds{1}_B\big(\frac{\tau(i)}{n}\big) \big) -
     \int_R u\big(x, \frac{i}{n}, \beta \big)u\big(\frac{i}{n}, y, \gamma \big)\, dxdy \Big| \nonumber\\
  <\ & \epsilon.   \nonumber
\end{align}
Here the last equality follows from the fact that $(\pi, \tau) \sim \n \times \mu_{n, q'_n}$ under $\p_n$, and the last inequality follows from (\ref{eq:L100d}).

Since $u(x, y, \beta)$ and $u(x, y, \gamma)$ are bounded on $[0, 1]\times[0, 1]$, by the definition of Riemann integral and the dominated convergence theorem, we have
\begin{align}
  &\textstyle\lim\limits_{n \to \infty}\frac{1}{n}\sum_{i=1}^{n}\int_R u\big(x, \frac{i}{n}, \beta \big)u\big(\frac{i}{n}, y, \gamma \big)\, dxdy   \label{eq:L100f}\\
  =\,&\textstyle \int_R \Big(\lim\limits_{n \to \infty}\frac{1}{n}\sum_{i=1}^{n}u\big(x, \frac{i}{n}, \beta \big)u\big(\frac{i}{n}, y, \gamma \big)\Big)dxdy \nonumber\\
  =\,&\textstyle \int_R \Big(\int_0^1 u(x, t, \beta) u(t, y, \gamma) \,dt\Big)\,dxdy \nonumber\\
  =\,&\textstyle \int_R \rho(x, y)\, dxdy. \nonumber
\end{align}
Hence, (\ref{eq:L100b}) follows from (\ref{eq:L100e}) and (\ref{eq:L100f}).

To show (\ref{eq:L100c}), similarly, by Lemma \ref{L13}, for any $\epsilon > 0$, there exists $N_2 > 0$ such that, for any $n > N_2$,
\[
\textstyle\max\limits_{\substack{ i \neq j\\
                                   i, j \in [n]}}
    \Big| \n \big(\mathds{1}_A\big(\frac{\pi(i)}{n}\big)\mathds{1}_A\big(\frac{\pi(j)}{n}\big)\big) -
     \n \big(\mathds{1}_A\big(\frac{\pi(i)}{n}\big)\big) \n \big(\mathds{1}_A\big(\frac{\pi(j)}{n}\big)\big) \Big| < \frac{\epsilon}{4},
\]
\[
\textstyle\max\limits_{\substack{ i \neq j\\
                                   i, j \in [n]}}
    \Big| \mu_{n, q'_n} \big(\mathds{1}_B\big(\frac{\tau(i)}{n}\big)\mathds{1}_B\big(\frac{\tau(j)}{n}\big)\big) -
     \mu_{n, q'_n} \big(\mathds{1}_B\big(\frac{\tau(i)}{n}\big)\big) \mu_{n, q'_n} \big(\mathds{1}_B\big(\frac{\tau(j)}{n}\big)\big) \Big|
     < \frac{\epsilon}{4}.
\]
Without loss of generality, assume $0 < \epsilon < 1$. Then, similar to (\ref{eq:L100d}), for any $n > N_2$ and any $ 1 \le i < j \le n$,
\begin{align}
  &\textstyle\Big|\cov_n\left(\mathds{1}_A\big(\frac{\pi(i)}{n}\big)\mathds{1}_B\big(\frac{\tau(i)}{n}\big),
    \mathds{1}_A\big(\frac{\pi(j)}{n}\big)\mathds{1}_B\big(\frac{\tau(j)}{n}\big) \right)\Big| \label{eq:L100g}\\
  =\,&\textstyle\Big|\E_n\left(\mathds{1}_A\big(\frac{\pi(i)}{n}\big)\mathds{1}_B\big(\frac{\tau(i)}{n}\big)
             \mathds{1}_A\big(\frac{\pi(j)}{n}\big)\mathds{1}_B\big(\frac{\tau(j)}{n}\big) \right) \nonumber\\
           &\textstyle\qquad\qquad\qquad   -
             \E_n\left(\mathds{1}_A\big(\frac{\pi(i)}{n}\big)\mathds{1}_B\big(\frac{\tau(i)}{n}\big)\right)
             \E_n\left(\mathds{1}_A\big(\frac{\pi(j)}{n}\big)\mathds{1}_B\big(\frac{\tau(j)}{n}\big)\right) \Big| \nonumber\\
  =\,&\textstyle\Big|\n \big(\mathds{1}_A\big(\frac{\pi(i)}{n}\big)\mathds{1}_A\big(\frac{\pi(j)}{n}\big)\big)
\mu_{n, q'_n} \big(\mathds{1}_B\big(\frac{\tau(i)}{n}\big)\mathds{1}_B\big(\frac{\tau(j)}{n}\big)\big) \nonumber\\
           &\textstyle\quad - \n \big(\mathds{1}_A\big(\frac{\pi(i)}{n}\big)\big) \n \big(\mathds{1}_A\big(\frac{\pi(j)}{n}\big)\big)
           \mu_{n, q'_n} \big(\mathds{1}_B\big(\frac{\tau(i)}{n}\big)\big) \mu_{n, q'_n} \big(\mathds{1}_B\big(\frac{\tau(j)}{n}\big)\big) \Big| \nonumber\\
  <\,&\textstyle \frac{\epsilon}{2}.   \nonumber
\end{align}
Here the second equality follows from the fact that $(\pi, \tau) \sim \n \times \mu_{n, q'_n}$ under $\p_n$, and the last inequality follows by triangle inequality. Specifically, if $0 \le a_1, a_2, b_1, b_2 \le 1$, $|a_1 - a_2| < \frac{\epsilon}{4}$ and $|b_1 - b_2| < \frac{\epsilon}{4}$, then we have
\[
|a_1 b_1 - a_2 b_2| \le |a_1 b_1 - a_2 b_1| + |a_2 b_1 - a_2 b_2| \le |a_1 - a_2| + |b_1 - b_2|< \frac{\epsilon}{2}.
\]
Here we choose
\begin{align*}
 &\textstyle a_1 = \n \big(\mathds{1}_A\big(\frac{\pi(i)}{n}\big)\mathds{1}_A\big(\frac{\pi(j)}{n}\big)\big),
 &\textstyle a_2 = \n \big(\mathds{1}_A\big(\frac{\pi(i)}{n}\big)\big) \n \big(\mathds{1}_A\big(\frac{\pi(j)}{n}\big)\big),\\
 &\textstyle b_1 = \mu_{n, q'_n} \big(\mathds{1}_B\big(\frac{\tau(i)}{n}\big)\mathds{1}_B\big(\frac{\tau(j)}{n}\big)\big),
 &\textstyle b_2 = \mu_{n, q'_n} \big(\mathds{1}_B\big(\frac{\tau(i)}{n}\big)\big) \mu_{n, q'_n} \big(\mathds{1}_B\big(\frac{\tau(j)}{n}\big)\big).
\end{align*}

Thus, for any $n > \max{(N_2, \frac{1}{\epsilon})}$,
\begin{align*}
  &\var_n (L_{\pi, \tau}(R))\\
 =\,&\textstyle\var_n \Big( \frac{1}{n}\sum_{i = 1}^{n}\mathds{1}_A\big(\frac{\pi(i)}{n}\big)\mathds{1}_B\big(\frac{\tau(i)}{n}\big) \Big)\\
 =\,&\textstyle\frac{1}{n^2}\sum\limits_{i = 1}^{n}\var_n\big(\mathds{1}_A\big(\frac{\pi(i)}{n}\big)\mathds{1}_B\big(\frac{\tau(i)}{n}\big)\big)\\
    &\textstyle\qquad\qquad\qquad + \frac{1}{n^2}\sum\limits_{\substack{ i \neq j\\
                                   i, j \in [n]}}\cov_n\big(\mathds{1}_A\big(\frac{\pi(i)}{n}\big)\mathds{1}_B\big(\frac{\tau(i)}{n}\big),
    \mathds{1}_A\big(\frac{\pi(j)}{n}\big)\mathds{1}_B\big(\frac{\tau(j)}{n}\big) \big)\\
 <\,&\textstyle\frac{1}{n^2}\cdot \frac{n}{4} + \frac{n(n-1)}{n^2}\cdot\frac{\epsilon}{2}\\
 <\,&\epsilon.
\end{align*}
The first inequality follows by (\ref{eq:L100g}) and the fact that the variance of any indicator function is no greater than $\frac{1}{4}$.
\end{proof}


Now we are in the position to prove Theorem \ref{M2}

\begin{proof}[Proof of Theorem \ref{M2}]
First of all, we make the following claim:

\textbf{Claim}: To prove Theorem \ref{M2}, it suffices to show the case when $f(x, y) = \mathds{1}_R(x, y)$, for any
$R = (x_1, x_2] \times (y_1, y_2] \subset [0, 1]\times[0, 1]$. 
This is because for any continuous function $f(x, y)$ and any $\epsilon > 0$, we can find a simple function $s(x, y)$ on $(0, 1]\times(0, 1]$ such that
\[
|f(x, y) - s(x, y)| < \frac{\epsilon}{3} \qquad \forall (x, y) \in (0, 1]\times(0, 1],
\]
where $s(x, y)$ is of the form
\[
s(x, y) = \sum_{j = 1}^{m}a_j\mathds{1}_{R_j}(x, y),
\]
with $R_j = \left(x_1^{(j)}, x_2^{(j)}\right]\times\left(y_1^{(j)}, y_2^{(j)}\right] \subset (0, 1]\times(0, 1]$ and $\{R_j\}_{j = 1}^{m}$ is a partition of $(0, 1]\times(0, 1]$.
Hence, we have
\begin{equation}\label{eq:M2b}
\textstyle\left|\, \frac{1}{n}\sum_{i=1}^{n}f\big(\frac{i}{n}, \frac{\tau \ccirc \pi(i)}{n}\big) -
\frac{1}{n}\sum_{i=1}^{n}s\big(\frac{i}{n}, \frac{\tau \ccirc \pi(i)}{n}\big) \,\right| < \frac{\epsilon}{3},
\end{equation}
and
\begin{align}
&\textstyle\left|\int_0^1\int_0^1 s(x, y) \rho(x, y)\,dxdy - \int_0^1\int_0^1 f(x, y) \rho(x, y)\,dxdy\, \right| \label{eq:M2c}\\
\le &\textstyle \int_0^1\int_0^1 \left|s(x, y) - f(x, y)\right| \rho(x, y)\,dxdy \nonumber\\
< &\textstyle\frac{\epsilon}{3}. \nonumber
\end{align}
Here we use the fact that, by Lemma \ref{L9},
\[
\textstyle\int_0^1\int_0^1 \rho(x, y) \, dxdy = 1.
\]
Thus, by (\ref{eq:M2b}), (\ref{eq:M2c}) and triangle inequality, we have
\begin{align*}
   &\textstyle\left|\, \frac{1}{n}\sum_{i=1}^{n}f\big(\frac{i}{n}, \frac{\tau \ccirc \pi(i)}{n}\big) -
  \int_0^1\int_0^1 f(x, y) \rho(x, y)\,dxdy\, \right| > \epsilon\\
 \Rightarrow&\textstyle \left|\, \frac{1}{n}\sum_{i=1}^{n}s\big(\frac{i}{n}, \frac{\tau \ccirc \pi(i)}{n}\big) -
    \int_0^1\int_0^1 s(x, y) \rho(x, y)\,dxdy \,\right| > \frac{\epsilon}{3}.
\end{align*}
Hence, we get
\begin{align*}
 &\textstyle \p_n\Big(\Big|\, \frac{1}{n}\sum_{i=1}^{n}f\big(\frac{i}{n}, \frac{\tau \ccirc \pi(i)}{n}\big) -
  \int_0^1\int_0^1 f(x, y) \rho(x, y)\,dxdy\, \Big| > \epsilon \Big)\\
 \le\, &\textstyle\p_n\Big(\Big|\, \frac{1}{n}\sum_{i=1}^{n}s\big(\frac{i}{n}, \frac{\tau \ccirc \pi(i)}{n}\big) -
    \int_0^1\int_0^1 s(x, y) \rho(x, y)\,dxdy \, \Big| > \frac{\epsilon}{3} \Big)\\
 \le\, &\textstyle\sum\limits_{j = 1}^{m} \p_n\Big(\Big|\, \frac{1}{n}\sum_{i=1}^{n}\mathds{1}_{R_j}\big(\frac{i}{n}, \frac{\tau \ccirc \pi(i)}{n}\big) -
    \int_{R_j} \rho(x, y)\,dxdy \, \Big| > \frac{\epsilon}{3 m |a_j|} \Big).
\end{align*}
Here the last inequality follows by the union bound. Therefore, to prove Theorem \ref{M2}, it suffices to show the case when $f(x, y) = \mathds{1}_R(x, y)$, with $R = (x_1, x_2]\times(y_1, y_2]$. In other words, we need to show that, for any $\epsilon > 0$,
\begin{equation}
\lim_{n \to \infty} \p_n \Big(\Big|L_{\tau \ccirc \pi}(R) - \int_{R}\rho(x, y)\, dxdy \Big| > \epsilon \Big) = 0.   \label{eq:M2a}
\end{equation}
Here, as defined in (\ref{eq:em}),
\[
L_{\tau \ccirc \pi}(R) \coloneqq \frac{1}{n}\sum_{i = 1}^{n} \mathds{1}_R\left( \frac{i}{n}, \frac{\tau \ccirc \pi(i)}{n} \right).
\]
Then, for any $\pi, \tau \in S_n$, we have
\begin{align*}
\left\{\big( i, \tau \ccirc \pi(i) \big) : i \in [n] \right\} &= \left\{\big(\pi^{-1}(\pi(i)), \tau(\pi(i)) \big) : i \in [n] \right\} \\
                                                   &= \left\{\big( \pi^{-1}(i), \tau(i) \big) : i \in [n] \right\}.
\end{align*}
The last equality follows since $\{\pi(i)\}_{i \in [n]} = \{i\}_{i \in [n]}$.
Thus, it follows that
\[
L_{\tau \ccirc \pi}(R) = L_{\pi^{-1}, \tau}(R), \qquad \forall R \in  \mathcal{B}_{[0, 1]\times[0,1]}.
\]
If $(\pi, \tau) \sim \m \times \mu_{n, q'}$, by Proposition \ref{P3}, $(\pi^{-1}, \tau) \sim \m \times \mu_{n, q'}$.
Thus, given $(\pi, \tau) \sim \p_n$, we have
\[
L_{\tau \ccirc \pi}(R) = L_{\pi^{-1}, \tau}(R) \overset{d}{=} L_{\pi, \tau}(R).
\]
That is $L_{\tau \ccirc \pi}(R)$ and $L_{\pi, \tau}(R)$ have the same distribution when $(\pi, \tau) \sim \p_n$.\\
Therefore, (\ref{eq:M2a}) follows by Lemma \ref{L100}.

\end{proof}

\section{Discussion and future works}
\begin{itemize}
\item[1.]
One question which arises in the proof of Theorem \ref{M2} is whether it also holds in more general settings. Specifically, suppose $\{X_n\}_{n = 1}^{\infty}$, $\{Y_n\}_{n = 1}^{\infty}$ are two independent sequences of random variables  with $X_n, Y_n \in S_n$ such that the empirical measures $\{L_{X_n}\}$ and $\{L_{Y_n}\}$ as defined in (\ref{eq:em}) converge weakly to the probability measures on the unit square with density $\rho_1(x, y)$ and $\rho_2(x, y)$ respectively. In other words, for any continuous function $f: [0,1]\times[0,1] \rightarrow \mathbb{R}$ we have
\begin{align*}
&\lim_{n \to \infty}\textstyle\p\left(\left|\, \frac{1}{n}\sum_{i=1}^{n}f\left(\frac{i}{n}, \frac{X_n(i)}{n}\right) -
  \int_0^1\int_0^1 f(x, y) \rho_1(x, y)\,dxdy\, \right| > \epsilon \right) = 0,\\
&\lim_{n \to \infty}\textstyle\p\left(\left|\, \frac{1}{n}\sum_{i=1}^{n}f\left(\frac{i}{n}, \frac{Y_n(i)}{n}\right) -
  \int_0^1\int_0^1 f(x, y) \rho_2(x, y)\,dxdy\, \right| > \epsilon \right) = 0.
\end{align*}
Then, under what conditions will the empirical measure of the product permutation $\{L_{Y_n \ccirc X_n}\}$ converge weakly to some probability measure? The proof of Theorem \ref{M2} shows that conditions analogous to Lemma \ref{L18} and Lemma \ref{L13} are sufficient for the weak convergence of $\{L_{Y_n \ccirc X_n}\}$ to the probability with density $$
\rho(x, y) \coloneqq \textstyle\int_0^1 \rho_1(x, t)\cdot \rho_2(t, y)\,dt.
$$
The proofs of Lemma \ref{L18} and Lemma \ref{L13} rely on the special properties of the Mallows measure. Specifically, Lemma \ref{M1} and Proposition \ref{P2} play the key roles in our proofs of Lemma \ref{L18} and Lemma \ref{L13}. We think a weaker condition should suffice to guarantee the weak convergence of  $\{L_{Y_n \ccirc X_n}\}$. However, there exist some ad hoc examples which indicate the intricacy of the problem. For instance, by Theorem 1.6 and Theorem 1.8 in \cite{hoppen}, for any probability measure on $[0, 1]\times[0, 1]$ with density $u(x, y)$ such that both marginal distributions are uniform measure on $[0, 1]$, there exists a sequence of permutations $\{\pi_n\}$ such that $L_{\pi_n}$ converges in distribution to $u(x, y)$. Suppose $u(x, y)$ is symmetric about the diagonal, i.e., $u(x, y) = u(y, x)$. By the symmetry of the graph of $\pi_n$ and $\pi_n^{-1}$, it follows that $L_{\pi_n^{-1}}$ also converges in distribution to $u(x, y)$. Define two independent sequences of random variables $\{X_n\}$ and $\{Y_n\}$ such that both $X_n$ and $Y_n$ have equal probability of being $\pi_n$ or $\pi_n^{-1}$. Then it is easy to see that both $L_{X_n}$ and $L_{Y_n}$ converge in distribution to $u(x, y)$, whereas $L_{Y_n \ccirc X_n}$ does not converge since with probability one half $Y_n \ccirc X_n$ is identity in $S_n$.

\item[2.] Given a sequence of independent random variables $\{X_n\}_{n = 1}^{\infty}$ where $X_n \sim \n$ and $\lim_{n \to \infty} n(1-q_n) = \beta$. Let $u_{\beta}$ denote the probability measure on the unit square with density $u(x, y, \beta)$ as defined in (\ref{eq:t1}). It is unknown to us whether Theorem \ref{T1} still holds in the following stronger sense.
\begin{equation}
\p\left(L_{X_n} \overset{d}{\rightarrow} u_{\beta}\right) = 1.
\end{equation}
In other words, does the weak convergence of the random measure $L_{X_n}$ to $u_{\beta}$ hold almost surely? Let $u_n$ denote the empirical measure of $n$ i.i.d.\newline samples from the distribution $u_{\beta}$, it is known that (cf. Theorem 11.4.1 in \cite{Dudley})
\begin{equation}
\p\left(u_n \overset{d}{\longrightarrow} u_{\beta}\right) = 1.
\end{equation}

\item[3.]
Another problem is parameter inference. Given a collection of permutations, which are samples from the product of two independent Mallows permutations with parameters $q$ and $q'$ respectively, how can we estimate parameters $q$, $q'$ based on the samples? For samples from a Mallows distribution, Proposition 1.12  in \cite{Mukherjee} shows the existence of a consistent estimator of the underlying parameter $q$. It is possible that the explicit density of the limiting distribution of the product of two Mallows permutations could be useful in deriving estimators for the underlying parameters.
\end{itemize}

\begin{acknowledgements}
I am grateful to my supervisor Nayantara Bhatnagar for her helpful advice and suggestions during the research as well as her guidance in the completion of this paper. 
\end{acknowledgements}

\bibliographystyle{spmpsci}      
\bibliography{sampart}   


\end{document}